\documentclass[reqno]{amsart}
\usepackage[utf8]{inputenc}
\usepackage{fullpage}
\usepackage{thm-restate}
\usepackage{graphicx} 
\usepackage[dvipsnames,svgnames,table]{xcolor}
\usepackage{enumitem}
\usepackage{amsmath,amssymb,amsthm}
\usepackage{mathtools}
\usepackage[T1]{fontenc}
\usepackage[final]{hyperref}
\hypersetup{
    colorlinks,
    urlcolor = cyan,
    linkcolor = black,
    citecolor = black,
    hypertexnames=false, 
}
\usepackage{booktabs}
\usepackage{multirow}
\usepackage[capitalize,nameinlink]{cleveref}
\newtheorem{theorem}{Theorem}[section]
\newtheorem{conjecture}[theorem]{Conjecture}
\newtheorem{lemma}[theorem]{Lemma}
\newtheorem{proposition}[theorem]{Proposition}
\newtheorem{corollary}[theorem]{Corollary}
\newtheorem*{theorem*}{Theorem}
\newtheorem{maintheorem}{Theorem}

\theoremstyle{definition}
\newtheorem{definition}[theorem]{Definition}
\theoremstyle{remark}
\newtheorem{remark}[theorem]{Remark}
\newtheorem{example}[theorem]{Example}

\numberwithin{equation}{section}

\usepackage[maxbibnames=99,sorting=nyt,giveninits,maxnames=10,backend=bibtex,block=space,isbn=false]{biblatex} 
\usepackage{doi}
\newcommand{\HH}{\mathcal{H}}
\newcommand{\symm}{\mathfrak{S}}
\newcommand{\CC}{\mathcal{C}}

\newcommand{\cc}{\mathfrak{c}}
\newcommand{\kk}{\mathbf{k}}

\newcommand{\C}{\mathbb{C}}
\newcommand{\syt}{{\sf SYT}}
\newcommand{\cont}{{\sf cont}}
\newcommand{\qcont}{{\sf qcont}}

\newcommand{\Z}{\mathbb{Z}}

\newcommand{\row}{ {\sf row}}
\newcommand{\col}{{\sf col}}

\newcommand{\M}{{\sf{M}}}
\newcommand{\F}{\mathbb{F}}
\newcommand{\MM}{\mathcal{M}}

\DeclareMathOperator{\sh}{sh}

\DeclareMathOperator{\GL}{GL}
\DeclareMathOperator{\Hilb}{Hilb}
\newcommand{\stat}{{\sf stat}}
\newcommand{\inv}{{\sf inv}}
\newcommand{\maj}{{\sf maj}}

\newcommand{\sfc}[1][n]{\mathsf{c}_{#1}}

\mathcode`\;="203B

\newcommand\qbinom[3]{{\begin{bmatrix} #1\\ #2 \end{bmatrix}_{#3}}}
\newcommand\smallqbinom[3]{{\left[\begin{smallmatrix} #1\\ #2 \end{smallmatrix}\right]_{#3}}}
\newcommand{\qint}[2][q]{\left\lbrack #2 \right\rbrack_{#1}}
\newcommand\qfact[1]{\left\lbrack #1 \right\rbrack!_q}

\newcommand\ncycle[1][n]{(1\, 2\, \ldots\, #1)}

\usepackage{tikz}
\title[Factorizations in Hecke algebras I]{Factorizations in Hecke algebras I: \\ long cycle factorizations and Jucys--Murphy elements}

\allowdisplaybreaks

\newcommand\YS{\begin{gathered}\begin{tikzpicture}[scale = .1]}
\newcommand\YSF{\begin{gathered}\begin{tikzpicture}[scale = .4]}

\newcommand\YBF[3]{\draw[fill = yellow] (#1,#2) rectangle (#1+1,#2+1) node[pos=.5] {\scriptsize $#3$};}
\newcommand\YF{\end{tikzpicture}\end{gathered}}

\author[Bastidas]{Jose Bastidas}
\address[J.~Bastidas]{LACIM, D\'epartement de Mathématiques, Universit\'e du Qu\'ebec \`a Montr\'eal, Canada}
\email{bastidas.math@proton.me}
\urladdr{https://bastidas-jose.codeberg.page/}

\author[Brauner]{Sarah Brauner}
\address[S.~Brauner]{Division of Applied Mathematics, Brown University, USA}
\email{sarahbrauner@gmail.com}
\urladdr{https://www.sarahbrauner.com/}

\author[Guay-Paquet]{Mathieu Guay-Paquet}
\address[M. Guay-Paquet]{LACIM, Universit\'e du Qu\'ebec \`a Montr\'eal, Canada}
\email{mathieu.guaypaquet@lacim.ca}
\urladdr{https://www.pointedset.ca/}

\author[Morales]{Alejandro H. Morales}
\address[A. H. Morales]{LACIM, D\'epartement de Mathématiques, Universit\'e du Qu\'ebec \`a Montr\'eal, Canada}
\email{morales\_borrero.alejandro@uqam.ca}
\urladdr{https://sites.google.com/view/ahmorales}

\author[Park]{GaYee Park}
\address[G. Park]{Department of Mathematics, Dartmouth College, USA}
\email{gayee.park@dartmouth.edu}
\urladdr{https://sites.google.com/view/gayeepark}

\author[Saliola]{Franco Saliola}
\address[F.~Saliola]{LACIM, D\'epartement de Mathématiques, Universit\'e du Qu\'ebec \`a Montr\'eal, Canada}
\email{saliola.franco@uqam.ca}
\urladdr{https://saliola.github.io/}


\begin{document}

\begin{abstract}
Given a permutation, there is a well-developed literature studying the number of ways one can factor it into a product of other permutations subject to certain conditions. We initiate the analogous theory for the type $A$ Iwahori--Hecke algebra by generalizing the notion of factorization in terms of the Jucys--Murphy elements. Some of the earliest and most foundational factorization results for the symmetric groups pertain to the long cycle. Our main results give $q$-deformations of these long cycle factorizations and reveal $q$-binomial, $q$-Catalan, and $q$-Narayana numbers along the way.
\end{abstract}
\maketitle
\section{Introduction}
Given a permutation $\sigma$ in the symmetric group $\symm_n$, there is an extensive and active body of research
(see e.g. \cite{bousquet2000enumeration,chapuy_Stump,GouldenJackson,GoupilSchaeffer,Jackson1}) that studies the question: how many different ways can $\sigma$ be written as a product of permutations satisfying a given condition? This is called a \emph{factorization} of $\sigma$. Permutation factorizations have diverse applications, for example to \emph{Hurwitz numbers} \cite{goulden1997transitive,goulden2005towards} which count ramified coverings of the sphere and are related to integrals over the moduli space of stable curves, as well as to the asymptotic expansions of certain integrals arising in random matrix theory and mathematical physics \cite{goulden2013monotone,goulden2013polynomiality}.

The origins of this field lie in the study of factorizations
of a specific and important element of the symmetric group: the long cycle $\sfc:= \ncycle$.
There is a variety of elegant formulas for different factorization statistics of $\sfc$, proved with diverse methods that span enumeration, generating functionology, and representation theory.

In this paper, we initiate the analogous study for the type $A$ Iwahori--Hecke algebra $\HH_n(q)$ (or Hecke algebra, for short), which is a $q$-deformation of the group algebra of the symmetric group $\C[\symm_n]$. The Hecke algebra has fundamental connections to geometric representation theory, knot theory, and even quantum computation and integrable probability. Our work is part of a larger movement (see e.g. \cite{axelrod2024spectrum, brauner2025q, clearman2016evaluations, clearwater2021total, GLTW24ratCatalan, Haiman, halverson1997iwahori, halverson2004q, kaliszewski2019bases, konvalinka2011generating, meliot2010products, ram1991frobenius, ram1996robinson, ram2024lusztig}) to uncover the \emph{combinatorics} of $\HH_n(q)$ in parallel to the cornucopia of results for the symmetric group.

One challenge of developing a factorization theory for $\HH_n(q)$ is that factorization in an algebra is more complex than in a group. For example, $\HH_n(q)$ has a natural basis $\{ T_w \}$ indexed by $w \in \symm_n$, but it is extremely rare for an element $T_w$ to be written as a product of other elements $T_{u_1} \cdots T_{u_k}$,
unless all the $u_k$ are simple transpositions.

Hence, our first task is to generalize the notion of factorization to something that makes sense in $\HH_n(q)$. To do so, we use the remarkable fact that many classical $\sfc$-factorization statistics can be rephrased algebraically as the coefficient of $\sfc$ in certain elements of the group algebra $\C[\symm_n]$. Given a symmetric function $f$, we consider the evaluation of $f$ at the \emph{Jucys--Murphy elements} $J_1, \ldots, J_n$ of the symmetric group algebra, and in particular the coefficient of $\sfc$ in this expression:
\[  [\sfc]f\big(J_1, \ldots, J_n\big) \quad \quad \textrm{ for } f \textrm{ a symmetric function}.\]
For the correct choice of $f$, this coefficient computes  the number of certain types of factorizations of $\sfc$; one can think of $f$ as dictating the parameters of the factorization.

Our notion of $\HH_n(q)$-factorization generalizes this perspective. We will compute the coefficient of $T_{\sfc}$---the $\HH_n(q)$-analogue of $\sfc$---in the evaluation of various symmetric functions at the $q$-Jucys--Murphy elements $J_1(q), \ldots, J_k(q)$ of $\HH_n(q)$:
\[ [T_{\sfc}]f\big(J_1(q), \ldots, J_n(q)\big) \quad \quad \textrm{ for }  f \textrm{ a symmetric function}.\]
Once more, the choice of $f$ determines the type of factorization being computed.

Our main results (stated in \cref{thm:intro_q_summary_thm} below) show that this notion of factorization in $\HH_n(q)$ is a good one, in the sense that we provide elegant $q$-deformations of several foundational factorization results for $\sfc$ (summarized in \cref{thm:introsymmetricsummary} below). In addition, our proofs provide a unifying framework with which to understand and further build the theory of factorizations, even when $q=1$. We expect our work here to be a jumping point for a robust $\HH_n(q)$-factorization theory, parallel to the classical story for $\symm_n$.

We highlight some of our results in \cref{fig:summary}.

\begin{table}[!h]
    \centering
    \begin{tabular}{m{80mm}m{30mm}l} \toprule
        number of factorizations of $\sfc$                                           & in $\symm_n$                              & $\HH_n(q)$-analogue                                                                                         \\ [3pt]  \midrule
        into $n-1$ transpositions                                                    & $n^{n-2}$                                        & $q^{-\binom{n}{2}} \qint{n}^{n-2}$                                                                        \\ [2ex]
        into $n-1$ transpositions (monotone)                                         & $C_{n-1}$                                        & $q^{-\binom{n}{2}}C_{n-1}(q)$                                                                          \\ [2ex]
        into two factors with $p+1$ and $n-p$ cycles                                 & $N_{n,p+1}$                                      & $q^{1-(p+1)(n-p)} N_{n,p+1}(q)$                                                                        \\ [2ex]
        into $m$ factors $\pi_i$ with $n-\lambda_i$ cycles for $\lambda \vdash n-1$  & $\frac{1}{n} \prod_{i=1}^m \binom{n}{\lambda_i}$ & $q^{\sum_i \binom{\lambda_i}{2}-\binom{n}{2}} \frac{1}{\qint{n}} \prod_{i} \smallqbinom{n}{\lambda_i}{q}$ \\ [3pt]
        \bottomrule \\
    \end{tabular}
    \caption{Summary of some factorization results in the symmetric group and the Hecke algebra.}
    \label{fig:summary}
\end{table}

\subsection{Prototypical factorization results in the symmetric group}\label{ssec:intro symmetric}
Our goal will be to generalize the following foundational examples of factorizations of the long cycle $\sfc = \ncycle \in \symm_n$.
\begin{itemize}[topsep=1ex, itemsep=1ex]
    \item Let $a(n; j)$ be the number of factorizations of $\sfc$ as a product of $j$ transpositions $\tau_1\cdots \tau_{j}$.
    \item Similarly, let $b(n; j)$ be the number of {\em monotone} factorizations of $\sfc$ as a product of $j$ transpositions $\tau_1\cdots \tau_{j}$, that is if $\tau_i=(a_i\;b_i)$ with $a_i < b_i$, then $b_1\leq b_2 \leq \cdots \leq b_{j}$.
    \item Lastly, let $c(n;p_1,p_2,\ldots,p_m)$ be the number of factorizations of $\sfc$ into $m$ permutations $\pi_1\cdots \pi_m$  where $\pi_i$ has $n-p_i$ cycles for $i=1,\ldots,m$. Note that $c(n;1^{j})=a(n;j)$.
\end{itemize}
For $(r_1, \ldots, r_m) \in \Z_{\geq 0}^m$, we also write $\M^{n - 1}_{(r_1,\ldots,r_m)}$ to be the number of tuples $(S_1,\ldots,S_m)$ of
    (not necessarily disjoint) subsets of $[{n - 1}]$ satisfying
    \[ S_1\cup \cdots \cup
    S_m =[{n - 1}] \quad \textrm{ and } \quad \# S_i=r_i \textrm{ for } i=1,\ldots,m. \]
The numbers $a(n;j), b(n;j)$ and $c(n;p_1, \ldots, p_m)$ have elegant formulas, summarized in \cref{thm:introsymmetricsummary}.
\begin{theorem}\label{thm:introsymmetricsummary}
    For $n\geq 1$ and $\sfc = \ncycle$, we have the following factorization results:
    \begin{enumerate}
        \item (Jackson, {\cite[p. 368]{Jackson0}}) 
The numbers $a(n;j)$ of factorizations of $\sfc$ into $j$ transpositions satisfy
\begin{align}
\sum_{j \geq n-1} a(n;j) \, t^{j} &= \frac{n^{n-2} t^{n-1}}{ \prod_{k=0}^{n-1} \left(1-n\left(\frac{n-1}{2}-k \right)t\right)}.  \label{eq:Jackson2}
\end{align}
In particular, (see Hurwitz, \cite{Hurwitz})
\begin{equation} \label{eq: tree case} a(n;n-1) = n^{n-2}.\end{equation} \smallskip
    \item (Matsumoto--Novak \cite[Thm. 2]{MN}) \label{thm:MatsumotoNovak}
The numbers $b(n;j)$ of monotone factorizations of $\sfc$ into $j$ transpositions satisfy
\begin{equation}\label{eq:Masumoto-Novak}
\sum_{j \geq n-1} b(n;j) \, t^{j} = \frac{C_{n-1} \, t^{n-1}}{\prod_{k=0}^{n-1}\left(1-k^2 t^2\right)},
\end{equation}
where $C_{n-1}$ is the $(n-1)$st-Catalan number
\[ C_{n-1}=\frac{1}{n}\binom{2n-2}{n-1}.\]
In particular, (see Biane \cite{Biane96} and Stanley \cite[Thm. 3.1]{StanleyPF})
\[ b(n;n-1) = C_{n-1}.\] \smallskip

\item (Jackson \cite[Thm. 4.3]{Jackson0})
The numbers $c(n;p_1,\ldots,p_m)$ of factorizations of $\sfc$
    into $m$ permutations $\pi_1\cdots \pi_m$, where $\pi_i$ has $n - p_i$ cycles,
    satisfy
    \begin{equation} \label{eq:JacksonSn}
    \begin{multlined}
   \sum_{0 \leq p_1,\ldots,p_m \leq n-1} c(n;p_1,\ldots,p_m) \ t_1^{n-p_1}\cdots t_m^{n-p_m} \\
    \qquad= \ n!^{m-1} \sum_{0\leq r_1,\ldots,r_m \leq n-1} \M^{n-1}_{(r_1,\ldots,r_m)} \binom{t_1}{n-r_1}\cdots \binom{t_m}{n-r_m}.
    \end{multlined}
    \end{equation}
In particular, (see Goulden--Jackson~\cite[Theorem 3.2]{GouldenJackson}) when $p_1 + p_2 + \cdots + p_m = n-1$,
\begin{equation} \label{eq:GJcacti}
c(n;p_1,\ldots,p_m)  = \frac{1}{n} \prod_{i=1}^m \binom{n}{p_i},
\end{equation}
and so (see B\'edard and Goupil \cite[Theorem 3.1]{BedardGoupil}), for $1 \leq k \leq n-1$,
\[ c\,(n;k,n-1-k) = N_{n,k+1} = \frac{1}{n} \binom{n}{k+1} \binom{n}{k}\]
where $N_{n,k+1}$ is the \emph{Narayana number.}
    \end{enumerate}
\end{theorem}

As we mentioned earlier, there is an interesting and uniform algebraic reformulation of the quantities $a(n;j), b(n;j)$ and $c(n;p_1,\ldots,p_m)$ (c.f. \cite{diaconis_greene,FerayJM,guay-paquet_thesis_2012}).
Recall that $\C[\symm_n]$ is the group algebra of the symmetric group, which has a linear basis given by permutations $w \in \symm_n$. For any $z \in \C[\symm_n]$ and $w \in \symm_n$, we write $[w]z$ to mean the coefficient of $w$ in $z$ when $z$ is expressed in this basis.

There are two keys to reformulating the factorization results for $\sfc$ algebraically. The first is the \emph{Jucys--Murphy elements} $J_1, \ldots, J_n$ of the symmetric group, where $J_k$ is the sum of transpositions:
\[ J_k = (1\ k) + (2 \ k) + (3 \ k) + \cdots + (k{-}1 \ k) \in \C[\symm_n].\]
The Jucys--Murphy elements are deeply connected to the representation theory of $\symm_n$, as described in \cref{sec:background-rep-theory}.
The second is the ring of symmetric functions $\Lambda$, along with the \emph{evaluation} of any $f \in \Lambda$ at the Jucys--Murphy elements:
\[ f(\Xi_n):=f\big(J_1, \ldots, J_n\big) \in \C[\symm_n]. \]
The previous factorization results can be understood as the coefficient of $\sfc$ in $f(\Xi_n)$ for different choices of $f \in \Lambda$. Namely,
\begin{align}
    a(n;j) &= [\sfc] \,\, {\color{RoyalBlue} e_{1^j}(\Xi_n)}
    \label{eq:a-reinterpreation}
    \\
    b(n;j) &= [\sfc] \,\, {\color{RoyalBlue} h_{j}(\Xi_n)}
    \label{eq:b-reinterpreation}
    \\
    c(n;p_1,\ldots,p_m) &= [\sfc] \,\, {\color{RoyalBlue} e_{(p_1,\ldots,p_m)}(\Xi_n)}
    \label{eq:c-reinterpreation}
\end{align}
where
\begin{itemize}
    \item
        $e_k$ is the $k$-th elementary symmetric function and
        $e_{(\alpha_1,\ldots,\alpha_m)}=e_{\alpha_1}\cdots e_{\alpha_m}$ for $\alpha_i \in \Z_{\geq 0}$;
    \item
        $h_j$ is the $j$-th complete symmetric function.
\end{itemize}
We make the right-hand-side of \cref{eq:a-reinterpreation}, \cref{eq:b-reinterpreation} and \cref{eq:c-reinterpreation} explicit in \cref{eq:a in sym alg}, \cref{eq: b in sym alg}, and \cref{eq: c in sym alg}, respectively.

\subsection{New factorization results for the Hecke algebra}
In this paper, we will give $q$-deformations of the quantities $a(n;j), b(n;j)$ and $c(n;p_1, \ldots, p_m)$ by computing the analogue of \cref{eq:a-reinterpreation}, \cref{eq:b-reinterpreation}, and \cref{eq:c-reinterpreation} in $\HH_n(q)$.

The relevant tools in this case will come from the evaluation of symmetric functions $f \in \Lambda$ in the \emph{$q$-Jucys--Murphy elements}
\begin{equation*}
    J_k(q) = q^{1-k} T_{(1\,k)} + q^{2-k} T_{(2\,k)} + \cdots + q^{-1} T_{(k{-}1\;k)} \in \HH_n(q).
\end{equation*}
These elements share many of the remarkable properties of their symmetric group counter-parts.
We denote the evaluation of a symmetric function $f$ at the
$q$-Jucys--Murphy elements by
\begin{equation*}
    f(\Xi_n(q)):=f\big(J_1(q),\ldots,J_n(q)\big) \in \HH_n(q).
\end{equation*}

Recall that we denote the long cycle by $\sfc = \ncycle \in \symm_n$,
and that $\HH_n(q)$ has a linear basis $\{ T_w \}$ indexed by permutations $w \in \symm_n$.
As in the case for the symmetric group, we would like to understand the coefficient
\[ [T_{\sfc}]f(\Xi_n(q))\]
for $f \in \Lambda$.
We define the following in $\HH_n(q)$:
\begin{align}
    a_q(n;j) &:= [T_{\sfc}] \,\, {\color{RoyalBlue} e_{1^j}\big(\Xi_n(q)\big)}
    \label{eq:a-q-reinterpreation}
    \\
    b_q(n;j) &:= [T_{\sfc}] \,\, {\color{RoyalBlue} h_{j}\big(\Xi_n(q)\big)}
    \label{eq:b-q-reinterpreation}
    \\
    c_q(n;p_1,\ldots,p_m) &:= [T_{\sfc}] \,\, {\color{RoyalBlue} e_{(p_1,\ldots,p_m)}\big(\Xi_n(q)\big)}
    \label{eq:c-q-reinterpreation}
\end{align}
As in the case of $\symm_n$,
we make the right-hand-side of \cref{eq:a-q-reinterpreation}, \cref{eq:b-q-reinterpreation} and \cref{eq:c-q-reinterpreation} explicit in \cref{eq:a-in-hecke}, \cref{eq:b in hecke}, and \cref{eq: c in hecke}, respectively.

Our results make use of the following standard $q$-analogues for $j, n, k \in \Z$ with $0 \leq k \leq n$:
\begin{equation*}
    \qint{j}:=(q^j-1)/(q-1),
    \qquad
    \qfact{n} = \qint{n} \qint{n-1} \cdots \qint{1},
    \qquad
    \qbinom{n}{k}{q}
    = \frac{ \qfact{n} }{ \qfact{k} \, \qfact{n-k} },
\end{equation*}
as well as the \emph{$q$-Catalan numbers} $C_{n}(q)$ and the \emph{$q$-Narayana numbers} $N_{n,k}(q)$
defined as
\begin{equation*}
    C_n(q) := \frac{1}{\qint{n+1}} \qbinom{2n}{n}{q}
    \qquad\text{and}\qquad
    N_{n,k}(q):=\dfrac{1}{\qint{n}}\qbinom{n}{k}{q}\qbinom{n}{k-1}{q}.
\end{equation*}
In addition, we define the following $q$-analogue of the binomial basis $\binom{t}{k} = \frac{1}{k!} \prod_{i=0}^{k-1} (t-i)$ of $\C[t]$:
\begin{equation*}
\binom{t}{k}_q := \frac{1}{\qfact{k}} \prod_{i=0}^{k-1} \big(t-q^{-i}\qint{i}\big) = \frac{1}{\qfact{k}} \prod_{i=0}^{k-1} \big(t + \qint{-i}\big).
\end{equation*}
Observe that $\binom{t}{k}|_{q=1} = \binom{t}{k}$ as a polynomial in $t$.

Finally, when $q$ is a power of a prime, let $\mathbb{F}_q$ be the finite field of cardinality $q$, and define $\M^{n}_{(r_1,\ldots,r_m)}(q)$ for $n,r_1,\ldots,r_m \in \Z_{\geq 0}$ as
\[\M^{n}_{(r_1,\ldots,r_m)}(q):= \# \Big \{ (V_1,\ldots,V_m): V_i \subset \mathbb{F}_q^{n} \textrm{ with } \dim(V_i) = r_i \textrm{ and } V_1 + \cdots + V_m = \mathbb{F}_q^{n} \Big\}.\]
One can show that $\M^{n}_{(r_1,\ldots,r_m)}(1) = \M^{n}_{(r_1,\ldots,r_m)}$, where $\M^n_{(r_1,\ldots,r_m)}$ is as defined in \cref{ssec:intro symmetric}; see \cref{cor:Mq-are-q-analogues}.

Our main result is a direct analogue of \cref{thm:introsymmetricsummary}.

\begin{maintheorem}\label{thm:intro_q_summary_thm}
    For $n\geq 1$, and $a_q(n;j), b_q(n;j)$ and $c_q(n; p_1, \ldots, p_m)$ as in \cref{eq:a-q-reinterpreation}, \cref{eq:b-q-reinterpreation}, and \cref{eq:c-q-reinterpreation}, we have the following:
    \begin{enumerate}
        \item (Proved as \cref{thm:tree conj}) The numbers $a_q(n;j)$ satisfy \begin{equation}
       \sum_{j \geq 0 } a_q(n;j) \, t^j
      = \dfrac{q^{-\binom{n}{2}} \ \qint{n}^{n-2} \ t^{n-1}}{\prod_{j=0}^{n-1} \left( 1 - \left( \tfrac{n - q^{-j}\qint{n}}{1-q} \right)t \right)}.
    \end{equation}
    In particular,
    \begin{equation}
       a_q(n;n-1)
        = q^{-\binom{n}{2}} \qint{n}^{n-2}.
    \end{equation}
    \smallskip

    \item (Proved as \cref{thm:cat conj}) The numbers $b_q(n;j)$ satisfy
        \begin{equation}
      \sum_{j \geq 0} b_q(n;j) \, t^j
       = \dfrac{q^{-\binom{n}{2}} \ C_{n-1}(q) \ t^{n-1}}{\prod_{j=-(n-1)}^{n-1} (1 - \qint{j}t)}.
    \end{equation}
    In particular,
    \begin{equation}\label{eq: q-Catalan}
     b_q(n;n-1)
        = q^{-\binom{n}{2}} \, C_{n-1}(q).
    \end{equation}
\smallskip
\item (Proved as \cref{thm:HeckeJacksonMs})
We have that

\begin{equation}\label{eq:intro-qMs}
\begin{multlined}
\sum_{0 \leq p_1,\ldots,p_m \leq n-1} c_q(n;p_1,\ldots,p_m) \ t_1^{n-p_1}\cdots t_m^{n-p_m} \\
      \qquad = \qfact{n}^{m-1} \quad \smashoperator{\sum_{0 \leq r_1,\ldots,r_m \leq n-1}} \quad q^{-\sum_i (n-r_i)r_i}  \, \M^{n-1}_{(r_1,\ldots,r_m)}(q) \, \binom{t_1}{n-r_1}_{q}\, \binom{t_2}{n-r_2}_q \cdots \binom{t_m}{n-r_m}_{q}.
        \end{multlined}
\end{equation}
In particular, when $p_1+p_2 + \cdots + p_m = n-1$,
\begin{equation} \label{intro:long-cycle-in-e-lambda-hecke-q}
c_q(n;p_1, \ldots, p_m)
= q^{\sum_i \binom{\lambda_i}{2}-\binom{n}{2}} \frac{1}{\qint{n}} \prod_{i} \qbinom{n}{\lambda_i}{q},
\end{equation}
and so for
$1 \leq k \leq n-1$ we have
\[
c_q(n;k,n-1-k)
= q^{1-(k+1)(n-k)} N_{n,k+1}(q).
\]
    \end{enumerate}
\end{maintheorem}

We remark that the numbers $\M^n_{(r_1,\ldots,r_m)}(q)$ are of independent interest since,
as a function of $q$,
they are in $\Z_{\geq 0}[q]$ (see \cref{prop:explicit-formula-Mq,prop:positive recurrence Ms}\footnote{This was conjectured by the authors and first proved by Matthew Bolan.}),
are equal modulo a power of $q$ to $q$-multinomial coefficients in certain cases (see \cref{cor:qMs m=2,lem:qMs-qmultinomial}), and appear to be unimodal (see \cref{rem:unimodality of qMs}).

\subsubsection{Evaluations via principal specializations}

The reader might recognize the
right-hand-side of \cref{intro:long-cycle-in-e-lambda-hecke-q} as a multiple
of the \emph{$q$-principal specialization} of $e_\lambda$.
It turns out this relationship holds more generally.
Indeed, we can use \cref{eq:intro-qMs}
to provide an explicit link between the coefficient of $T_{\sfc}$ in
$f(\Xi_n(q))$ and the $q$-principal specialization of $f$ for \emph{any}
homogeneous symmetric function $f$. This extends an observation of Irving
from the symmetric group algebra to the Hecke algebra.

\begin{maintheorem}[proved as \cref{thm: key lemma q reciprocity} and \cref{thm:other-symmetric-functions}]\label{thm:intro_reciprocity}
    Let $f$ be a homogeneous symmetric function of degree $n-1$. Then
    \begin{eqnarray}
        [T_{\sfc}] \, f(\Xi_n(q))
        & = &
        \frac{q^{-\binom{n}{2}}}{\qint{n}} \ ps_f(n; q)
    \end{eqnarray}
    where $ps_f(n; q)$ is the \emph{$q$-principal specialization} of $f$
    defined as
    \begin{equation*}
        ps_f(n; q) := f\big(1, q, q^2, \ldots, q^{n-1}\big).
    \end{equation*}
    Thus, for any partition $\lambda \vdash n - 1$, we have
      \begin{eqnarray}
        [T_{\sfc}] \,\, h_{\lambda}\big(\Xi_n(q)\big)
        & = &
        \frac{q^{-\binom{n}{2}}}{\qint{n}} \, \prod_{i = 1}^{\ell(\lambda)} \qbinom{n+\lambda_i-1}{\lambda_i}{q},
        \\\null
        [T_{\sfc}] \,\, p_{\lambda}\big(\Xi_n(q)\big)
        & = &
        \frac{q^{-\binom{n}{2}}}{\qint{n}} \, \prod_{i=1}^{\ell(\lambda)} \frac{1-q^{n\lambda_i}}{1-q^{\lambda_i}},
        \\\null
        [T_{\sfc}] \,\, s_{\lambda}\big(\Xi_n(q)\big)
        & = &
        \frac{q^{b(\lambda)-\binom{n}{2}}}{\qint{n}} \, \prod_{(i,j)\in \lambda} \frac{1-q^{n+j-i}}{1-q^{h(i.j)}},
    \end{eqnarray}
    where $b(\lambda)=\sum_{i=1}^{\ell(\lambda)} (i-1)\lambda_i$,
    and $h(i,j) = \lambda_i+\lambda'_j-i-j+1$ is the hook-length of the cell $(i,j)$ in $\lambda$.  
\end{maintheorem}

\subsection{Other $q$-analogues}
Our work is, to the best of our knowledge, the first that studies $\HH_n(q)$-analogues of the factorization results in the symmetric group. However, there is an important, distinct $q$-analogue of the factorization results for $a(n;j)$ and $c(n;p_1, \ldots, p_m)$ in the finite general linear group
$\GL_n(\mathbb{F}_q)$ of $n\times n$ invertible matrices over the finite field
$\mathbb{F}_q$. Lewis--Reiner--Stanton in \cite{lewis2014reflection} adapted the definition of $a(n; j)$ to $\GL_n(\mathbb{F}_q)$. This led to analogous results for $c(n;p_1,\ldots,p_m)$ \cite{HLR,LM}. There is no known analogue of the monotone case $b(n;j)$ for $\GL_n(\mathbb{F}_q)$.

One benefit of their $\GL_n(\mathbb{F}_q)$ results is that they are true factorizations, in the original sense of the word. On the other hand, unlike our results, their formulas do not specialize to the classical $a(n;j)$ and $c(n;p_1, \ldots, p_m)$ at $q=1$.
However, there is a close link between $\HH_n(q)$ and $\GL_n(\mathbb{F}_q)$ similar in flavor to the Schur--Weyl duality between $\symm_n$ and $\GL_n(\C)$ (see e.g. \cite{halverson1999bitraces}). Thus it would be worthwhile and interesting to understand the relationship between our factorization results and those for $\GL_n(\mathbb{F}_q)$.

\subsection{Outline}

\cref{sec:background-rep-theory} will review some background on the
representation theory of $\symm_n$ and $\HH_n(q)$, including the Jucys--Murphy elements and their connections to the centers of the respective algebras $Z(\C[\symm_n])$ and $Z(\HH_n(q))$.
\cref{sec:main-results} presents the proof of \cref{thm:intro_q_summary_thm} and \cref{thm:intro_reciprocity}, with the
principal techniques we use highlighted in
\cref{ssec:isotypic-projectors-technique} and
\cref{ssec:evaluations-other-symmetric-functions}.
In \cref{sssec:interpretation-and-combinatorics-of-Ms} we study the combinatorics of the polynomials $\M^{n}_{(r_1,\ldots,r_m)}(q)$.
We end with with some final remarks in \cref{sec:final-remarks}.

\section*{Acknowledgments}
We are extremely grateful to Nadia Lafrenière for illuminating discussions and input at an earlier stage in this project. We thank Amarpreet Rattan for an inspiring talk about \cite{LothRattan} that kickstarted this project, and John Irving for pointing out the connection between principal specializations and factorizations of the symmetric group, which led to \cref{thm:intro_reciprocity}. We thank Matthew Bolan for communicating the recursion in \cref{prop:positive recurrence Ms}. We also thank François Bergeron, Guillaume Chapuy, Joel Kamnitzer, Joel Lewis,  Arun Ram,  Vic Reiner, Mark Skandera, Hunter Spink, and Nathan Williams for helpful comments and suggestions.  This work started from a working group at LACIM in Montréal and the research was facilitated by computer experiments using Sage~\cite{sagemath} and its algebraic combinatorics features
developed by the Sage-Combinat community~\cite{Sage-Combinat}.

The second named author is partially supported by the NSF MSPRF DMS-2303060. The fourth named author acknowledges the support of the Natural Sciences and Engineering Research Council of Canada (NSERC) funding reference number RGPIN-2024-06246, and was partially supported by the NSF grant DMS-2154019. The sixth named author acknowledges the support of the Natural Sciences and Engineering Research Council of Canada (NSERC) funding reference number RGPIN-2023-04476.

\section{Background: character theory and centers of $\kk[\symm_n]$ and $\HH_n(q)$}
\label{sec:background-rep-theory}

In this section we review some relevant aspects of the character theory of the symmetric group $\symm_n$, and its connections to the Jucys--Murphy elements and the center $Z(\kk[\symm_n])$ of the group algebra $\kk[\symm_n]$. We then give an analogous treatment for the type $A$ Iwahori--Hecke algebra $\HH_n(q)$. Henceforth, we assume that we are working in an algebraically closed field $\kk$ of characteristic zero.

\subsection{The symmetric group}

\subsubsection{Character theory of $\kk[\symm_n]$}\label{section:charactertheory for symmetric group}
The semisimple representation theory of the symmetric group is a beautiful and well-developed subject. We assume the reader has a knowledge of basic representation theory of finite groups, and review here only the aspects relevant for our study; see \cite{sagan2013symmetric, EC2} for excellent introductions to the representation theory of the symmetric group and symmetric functions, respectively.

In what follows, there are two useful perspectives to describe permutations in $\symm_n$.
\begin{enumerate}
    \item In \emph{cycle notation}, so that $w \in \symm_n$ is written as a product of disjoint cycles. The \emph{cycle type} of $w$ is the partition $\mu \vdash n$ determined by the size of these disjoint cycles.
    \item As a \emph{reduced expression} in the simple transpositions $s_i := (i\,\,i{+}1)$, so that $w = s_{i_1} \cdots s_{i_\ell}$. Note that the sequence $i_1, \ldots, i_\ell$ is not unique, but the number of simple transpositions in a minimal-length reduced expression for $w$ is; this is called the \emph{length} of $w$ and is written as $\ell(w)$.
\end{enumerate}

In the symmetric group, two elements are conjugates if and only if they have the same cycle type.
Thus, conjugacy classes are also indexed by integer partitions: the conjugacy class $\CC_\mu$ consists of all permutations with cycle type $\mu$.
We write
\begin{equation}
\label{eq:def-conj-cl-sum}
\cc_\mu:= \sum_{w \in \CC_\mu} w
\end{equation}
for the sum in $\kk[\symm_n]$ of the elements in the conjugacy class $\CC_\mu$, which we call a \emph{conjugacy class sum}.

\begin{example}
    \label{example:symm3-conjugacy-class-sums}
    For $n = 3$, there are $3$ partitions of $n$: $(1, 1, 1)$, $(2, 1)$ and $(3)$. 
    The conjugacy classes of $\symm_3$ are
    \begin{align*}
        \CC_{(1, 1, 1)} =  \left\{ \epsilon \right\} ,
        \qquad
        \CC_{(2, 1)} =  \left\{ (1\,2), (1\,3), (2\,3) \right\},
        \qquad \textrm{ and } \qquad
        \CC_{(3)} =  \left\{ (1\,2\,3), (1\,3\,2) \right\}.
    \end{align*}
    The corresponding conjugacy class sums are
    \begin{align*}
        \cc_{(1,1,1)}  = \epsilon,
        \qquad
        \cc_{(2,1)} =   (1\,2) + (2\,3) + (1\,3),
         \qquad \textrm{ and } \qquad
        \cc_{(3)} = (1\,2\,3) + (1\,3\,2).
    \end{align*}
\end{example}

The irreducible representations of the symmetric group are also indexed by integer partitions $\lambda \vdash n$.
We write $S^\lambda$ for the \emph{Specht module} corresponding to $\lambda$, which is a concrete realization of the corresponding irreducible representation.
We write the character of the Specht module $S^\lambda$ as $\chi^\lambda$.

Since characters are constant on conjugacy classes, the value of $\chi^\lambda(w)$ depends only on the cycle type of $w$, and so we will abuse notation and write $\chi^\lambda(\mu)$ to mean $\chi^\lambda(w)$ for any $w \in \CC_\mu$.
Linearly extending $\chi^\lambda$ to the group algebra $\kk[\symm_n]$, we have
\begin{equation}\label{eq:conj.class.sum}
\chi^\lambda(\cc_\mu) = \sum_{w \in \CC_\mu} \chi^\lambda(w) = (\# \CC_\mu) \, \chi^\lambda(\mu),
\end{equation}
Note that $\chi^\lambda(1^n)$ (i.e. the character of $S^\lambda$ evaluated at the identity of $\symm_n$) is the dimension of $S^\lambda$, which is $f^\lambda$, the number of \emph{standard Young tableaux of shape $\lambda$}. Alternatively, the \emph{hook length formula} gives another expression for $f^\lambda$. Let $\lambda'$ be the conjugate partition of $\lambda$. Then
\begin{equation} \label{eq:hooklength}
    f^\lambda = \frac{n!}{\prod_{(i,j)\in \lambda} h(i,j)},
\end{equation}
where $h(i,j) = \lambda_i+\lambda'_j-i-j+1$ is the hook-length of the cell $(i,j)$ in $\lambda$. 

More generally, it is a well-known and celebrated result is that the value $\chi^\lambda(\mu)$ can be computed combinatorially from the data of $\lambda$ and $\mu$; this is known as the \emph{Murnaghan--Nakayama rule} (see \cite[\S 7.17]{EC2} and the references therein).

The conjugacy class sum $\cc_\mu$ and characters $\chi^\lambda$ play an important role in the study of the center of the group algebra $\kk[\symm_n]$, which we shall denote by $Z(\kk[\symm_n])$. In particular, there are two notable bases of $Z(\kk[\symm_n])$. The first is given by the conjugacy class sums
\[ \cc:= \big \{ \cc_\mu \big \}_{\mu \vdash n}.\]
The second is a basis of orthogonal idempotents:
\[ \pi := \big \{ \pi_\lambda \big \}_{\lambda \vdash n},\]
meaning that 
\[
    \pi_\lambda \pi_\mu = \delta_{\lambda,\mu} \pi_\lambda,
    \qquad\text{where~}
    \delta_{\lambda, \mu} = \begin{cases}
        1, & \lambda = \mu,\\
        0, & \lambda \neq \mu.
    \end{cases}
\]
They are also \emph{complete}, so that $\sum_{\lambda \vdash n} \pi_\lambda = \epsilon$, the identity element of the algebra $\mathbf{k}[\symm_n]$.

The $\pi_\lambda$ can be expressed using the characters $\chi^\lambda$:
\begin{equation} \label{eq:isotypic-projectors}
    \pi_\lambda:= \frac{f^\lambda}{n!} \sum_{w \in \symm_n} \chi^{\lambda}(w) w^{-1}=
    \frac{f^\lambda}{n!} \sum_{w \in \symm_n} \chi^{\lambda}(w) w  =
    \frac{f^\lambda}{n!} \sum_{\mu \vdash n} \chi^\lambda(\mu) \cc_\mu.
\end{equation}

\begin{example}
    Let $\lambda = (1, 1, 1) = (1^3)$.
    The Specht module $S^{(1^3)}$ corresponds to the sign representation, whose character $\chi^{(1^3)}(w)$ is given by the parity of the length of $w$.
    Thus,
    \begin{align*}
    \cc_{(1^3)} &= \epsilon & \chi^{(1^3)}({(1^3)}) &= 1\\
    \cc_{(2,1)} &= (1\,2) + (2\,3) + (1\,3) & \chi^{(1^3)}({(2,1)}) &= -1\\
    \cc_{(3)} &= (1\,2\,3) + (1\,3\,2) & \chi^{(1^3)}({(3)}) &= 1.
    \end{align*}
    Then the corresponding orthogonal idempotent is:
    \[
        \pi_{(1^3)}
        =\tfrac{1}{3!} \bigg(\epsilon - \Big((12) + (23) + (13)\Big)  + \Big((123) + (132)\Big)\bigg).
    \]
\end{example}

\begin{remark}
In the language of representation theory, the idempotents $\pi_\lambda$ are known as \emph{isotypic projectors}, because they serve as projectors from any $\symm_n$-module $M$ to the copies of $S^\lambda$ inside of $M$. In other words,
\[ \pi_\lambda M \cong \left( S^\lambda \right)^{\oplus m_\lambda},\]
where $m_\lambda$ is the \emph{multiplicity} of $S^\lambda$ in $M$.
\end{remark}

The characters of $\symm_n$ provide elegant formulas to move between the $\cc$ and $\pi$ bases:
\begin{align}
    \pi_\lambda &= \frac{f^\lambda}{n!} \sum_{\mu \vdash n} \chi^{\lambda}(\mu) \cc_\mu \label{eq:pi.to.c}\\
    \cc_\mu &= \sum_{\lambda \vdash n} \frac{\chi^{\lambda}(\cc_\mu)}{f^\lambda} \pi_\lambda \label{eq:c.to.pi}
\end{align}
Of course, \cref{eq:pi.to.c} is essentially the definition of $\pi_\lambda$.
One way of interpreting \cref{eq:pi.to.c} and \cref{eq:c.to.pi} is as describing the rows and columns of the character table of $\symm_n$, respectively.
Those familiar with symmetric functions might observe that these change-of-basis formulas bear a striking resemblance to those for the Schur and power symmetric functions; this can be explained by studying the Frobenius characteristic map.

\subsubsection{The Jucys--Murphy elements}\label{section.JMforsymmetricgroup}

As discussed in the introduction, a key component of our story will be the \emph{Jucys--Murphy elements} of the symmetric group, which provide a link between representation theory and enumeration results.

\begin{definition}\label{def:JM}
For any $1 \leq k \leq n$, define $J_k \in \kk[\symm_n]$ to be the sum of the transpositions:
\[ J_k := (1 \, k) + (2 \, k) + \cdots + (k{-}1 \,\, k) \in \kk[\symm_n].\]
By convention, $J_1 = 0$.
\end{definition}

One surprising but important property of the $J_k$ is that they pair-wise commute \cite{jucys1974symmetric},
and thus $\kk[\symm_n]$ decomposes as a sum of the common eigenspaces of the $J_k$ acting by (left) multiplication\footnote{Note that everything can also be rephrased in terms of right multiplication, as long as one is consistent.}.
To describe these eigenspaces, we will use the combinatorial notion of \emph{content}, whose definition we build up below.

Given a partition $\lambda = (\lambda_1, \ldots, \lambda_\ell)$,
the \emph{Young diagram} (or \emph{Ferrers diagram}) of $\lambda$ is the collection of left-justified boxes where the $i$-th row has $\lambda_i$ boxes.
We adopt English notation: row $(i+1)$ is below row $i$. Write $\syt(\lambda)$ to be the set of \emph{standard Young tableaux} of shape $\lambda$, which are bijective labelings of the Young diagram of $\lambda$ with the set $[n]= \{1, \ldots, n \}$ such that labels are increasing across rows (west to east) and down columns (north to south). The collection of all standard Young tableaux of size $n$ is denoted $\syt_n$.

Given a Young diagram of shape $\lambda$, index the columns by $1$, $2$, $3$, and so on, from west to east, and the rows from north to south. Then for any $T \in \syt(\lambda)$ and $k \in [n]$, let $\row_k(T)$ and $\col_k(T)$ denote the row and column index, respectively, of the box labeled $k$ in $T$.

\begin{definition}
For any partition $\lambda \vdash n$, standard Young tableau $T \in \syt(\lambda)$, and $k \in [n]$, the \emph{content} of $T$ at $k$ is
\[ \cont_k(T)= \col_k(T) - \row_k(T).\]
The \emph{multiset of contents of $T$} is
\[ \cont(T) = \big\{\!\!\big\{ \cont_k(T) : 1 \leq k \leq |T| \big\}\!\!\big\}, \]
and we define
\[ \cont(\lambda) := \cont(T) \text{~for any $T \in \syt(\lambda)$.} \]
\end{definition}

\begin{example}
    \label{ex:contents}
    Below we show the Young diagram of the partition $\lambda = (5,4,2,2,1)$. Each box is labeled by the content at that box:
    \[
    \YSF
    \YBF{0}{0}{0}\YBF{1}{0}{1}\YBF{2}{0}{2}\YBF{3}{0}{3}\YBF{4}{0}{4}
    \YBF{0}{-1}{-1}\YBF{1}{-1}{0}\YBF{2}{-1}{1}\YBF{3}{-1}{2}
    \YBF{0}{-2}{-2}\YBF{1}{-2}{-1}
    \YBF{0}{-3}{-3}\YBF{1}{-3}{-2}
    \YBF{0}{-4}{-4}
    \YF
    \]
    Thus in this case,
    \[
        \cont\big((5,4,2,2,1)\big) = \big\{\!\!\big\{ -4,-3,-2,-2,-1,-1,0,0,1,1,2,2,3,4 \big\}\!\!\big\}.
    \]
\end{example}

Recall that the $J_k$ pairwise commute, and thus share an eigenbasis. We will be interested in building a family of elements $\pi_T$ in $\kk[\symm_n]$ with the following properties:
\begin{enumerate}
\item they are indexed by standard Young tableaux $T$ of size $n$;
\item they project $\kk[\symm_n]$ onto the simultaneous eigenspaces of the elements $J_1, \ldots, J_n$ acting on $\kk[\symm_n]$;
\item and the element indexed by $T$ is an eigenvector for $J_k$ with corresponding eigenvalue $\cont_k(T)$.
\end{enumerate}

In fact, the $\pi_T$ can be defined so that they satisfy the above properties by construction, using a process called Lagrange interpolation. To avoid unnecessary technicalities, we will not define the $\pi_T$ here, but refer the reader instead to \cite[Definition 2.10]{axelrod2024spectrum}.
Instead, we state here their properties, which are far more relevant.
\begin{proposition}\label{lem:interaction-of-Jucys--Murphys-with-pts}
The elements $\pi_T \in \kk[\symm_n]$ for $T \in \syt_n$ satisfy the following properties:
\begin{enumerate}
     \item The collection of $\pi_{T}$ for $T \in \syt_n$
form a family of mutually orthogonal, complete idempotents:
     \[ \pi_{T} \  \pi_{Q} = \delta_{T, Q} \pi_{T} \quad \quad \textrm{ and } \quad \quad \sum_{T \in \syt_n} \pi_{T} = 1.\]
    \item The $\pi_T$ refine the isotypic projectors $\pi_\lambda$:
    \[ \pi_\lambda = \sum_{T \in \syt(\lambda)} \pi_T. \]
    \item Each $\pi_T$ is a simultaneous (left) eigenvector for $J_k$ for $1 \leq k \leq n$:
    \[  J_k \ \pi_T   = \cont_k(T) \ \pi_T.\]
\end{enumerate}
\end{proposition}
The $\{ \pi_T \}_{T \in \syt(\lambda)}$ form a subset of the joint $\{ J_1, \ldots, J_n \}$-eigenbasis of $\kk[\symm_n]$ known as the \emph{Young seminormal forms} (see \cite{garsia2020young, jucys1974symmetric,murphy1981new,young1977quantitative}).

\begin{example}
Recall that $J_1 = 0$, so $J_1 \pi_T = 0$ for all $T \in \syt_n$. This matches the formula above since $\cont_1(T) = 0$ for all $T \in \syt_n$.
Below we show the action of the remaining Jucys--Murphy elements on the $\pi_T$ indexed by the two standard Young tableaux of shape $(2,1)$.

Let
\begin{align*}
    T =  \YSF\YBF{0}{0}{1}\YBF{1}{0}{2}\YBF{0}{-1}{3}\YF \quad \quad \quad \quad Q = \YSF\YBF{0}{0}{1}\YBF{1}{0}{3}\YBF{0}{-1}{2}\YF
\end{align*}
Then
\begin{align*}
    J_2 \cdot \pi_T &= \cont_2(T) \cdot \pi_T = \mathbf{1} \cdot \pi_T \qquad & J_2 \cdot \pi_Q = \cont_2(Q) \cdot \pi_Q = \mathbf{-1} \cdot \pi_Q \\
    J_3 \cdot \pi_T &= \cont_3(T) \cdot \pi_T = \mathbf{-1} \cdot \pi_T \qquad & J_3 \cdot \pi_Q = \cont_3(Q) \cdot \pi_Q = \mathbf{1} \cdot \pi_Q \\
\end{align*}
\end{example}

\subsubsection{Symmetric functions and the center}

The center of the symmetric group algebra and the Jucys--Murphy elements come together beautifully via symmetric functions.

We assume basic familiarity with the theory of symmetric functions.
Let $\Lambda$ be the ring of symmetric functions over $\kk$ and $\Lambda_n$ be the subspace of homogeneous symmetric functions of degree $n$.
We will frequently restrict a symmetric function $f \in \Lambda$ to a finite set of variables $x_1, \ldots, x_n$,
in which case we write
\[ f(x_1, \ldots, x_n):= f(x_1, \ldots, x_n, 0, 0, \ldots). \]
In this case, $f(x_1, \ldots, x_n)$ is a symmetric polynomial.

We will be interested in the following bases of $\Lambda$, following the notation from \cite[Chapter 7]{EC2}. Let $m_{\lambda}$, $e_k$, $h_k$, $p_k$, and $s_\lambda$ denote the {\em monomial}, {\em elementary}, {\em complete homogeneous}, {\em power sum}, and {\em Schur} symmetric functions, respectively.

For our purposes, the key connection between symmetric functions and the Jucys--Murphy elements $J_k$ is the following: given $f \in \Lambda$, consider the evaluation of $f$ at the Jucys--Murphy elements:
\begin{equation}
    f(\Xi_n):= f\big(J_1, \ldots, J_n\big) \in \kk[\symm_n]. \label{eq:sym.evaluationmap}
\end{equation}
We call such an evaluation the \emph{content evaluation of $f$}, for reasons that will be clear shortly.

The content evaluation has the following powerful characterization:
\begin{theorem}[Jucys \cite{jucys1974symmetric}] \label{thm:isomorphism.center.evaluation}
For any $n \geq 1$, the element $z \in \kk[\symm_n]$ is in the center $Z(\kk[\symm_n])$ if and only if $z$ can be expressed as
    $ z = f(\Xi_n) $
    for some $f \in \Lambda$.
    In other words,
    \[Z(\kk[\symm_n]) = \big\{ f(\Xi_n) \,:\, f \in \Lambda \big\}.\]
\end{theorem}

Though beyond the scope of our current work, there is an elegant method of studying the center of all the symmetric group algebras simultaneously via the \emph{Farahat--Higman algebra}; see Farahat--Higman's work \cite{farahat1959centres}, as well as a more contemporary discussion in \cite[\S 3]{ryba2021stable}. The content evaluation map yields an isomorphism between $\Lambda$ and the Farahat--Higman algebra. Work of Corteel--Goupil--Schaeffer describes a basis of $\Lambda$ whose elements have content evaluations corresponding to a basis of this algebra \cite{CorteelSchaeffer}.
\begin{example}[$e^j_1(\Xi_n)$]\label{ex:e1evaluation}
Suppose $f = e_1$. Then
\[ e_1(\Xi_n) = e_1(J_1, \ldots, J_n) = \sum_{1 \leq i < j \leq n} (i\,j) = \cc_{2,1^{n-2}}. \]

Now consider
\[ e_1^{n-1}(\Xi_n) = \big( e_1(\Xi_n) \big)^{n-1} =
\Big( \hskip3pt \smashoperator{\sum_{1\leq i<j\leq n}} \ (i\,j) \hskip3pt \Big)^{n-1}. \]
Since every term in the above expression is a product of $n-1$ transpositions,
the coefficient of $\sfc$ is precisely number of ways to obtain  $\sfc$ as a product of $n-1$ transpositions---in other words, $a(n;n-1)$ from the introduction.

More generally,
\begin{align}\label{eq:a in sym alg}
  e_1^k(\Xi_n) =  \Big( \hskip3pt \smashoperator{\sum_{1\leq i<j\leq n}} \ (i\,j) \hskip3pt \Big)^k,
\end{align}
and $a(n;k)$ is the coefficient of $\sfc$ in \cref{eq:a in sym alg}.
\end{example}

\begin{example}[$h_k(\Xi_n)$]
Now consider $f = h_j$. We have
\begin{align}
    h_j(\Xi_n) =  \smashoperator{\sum_{\substack{
                    2\leq t_1\leq \cdots \leq t_j\leq n \\
                    s_1<t_1,\ldots,s_j<t_j
            }}
        }
        \hskip5pt
        (s_1\,t_1) \cdots (s_j\,t_j) \label{eq: b in sym alg}
\end{align}
Thus again the coefficient of $\sfc$ in \cref{eq: b in sym alg} is precisely the number of ways to write $\sfc$ as a product of $j$ transpositions $\tau_1\cdots \tau_{j}$ such that if $\tau_i=(a_i\;b_i)$ with $a_i < b_i$, then $b_1\leq b_2 \leq \cdots \leq b_{j}$. In other words, this coefficient is $b(n;j)$, the number of monotone factorizations of $\sfc$ into $j$ transpositions.
\end{example}

\begin{example}[$e_{(p_1,\ldots,p_m)}(\Xi_n)$]
Consider $f = e_k$.
Using a similar argument to \cref{ex:e1evaluation}, Jucys showed in \cite[\S 3]{jucys1974symmetric} that
\begin{equation} \label{eq:ekevaluation}
    e_k(\Xi_n) =  \sum_{\substack{\lambda  \vdash n \\k = n-\ell(\lambda)}} \ \cc_\lambda \in \kk[\symm_n].
\end{equation}
It follows that for positive integers $(p_1, \ldots, p_m)$,
\begin{align}
    e_{(p_1, \ldots, p_m)}(\Xi_n)
    & = e_{p_1}(\Xi_n) \ e_{p_2}(\Xi_n) \cdots e_{p_m}(\Xi_n) \\
    & =
    \hskip-10pt
    \sum_{\substack{\mu^{(i)} \vdash n \\ \ell(\mu^{(i)}) = n-p_i}}
    \hskip-10pt
        {\cc}_{\mu^{(1)}} \cdots {\cc}_{\mu^{(m)}}. \label{eq: c in sym alg}
\end{align}
Since the sum is over all partitions $\mu^{(i)}$ of $n$ of length $n-p_i$,
the coefficient of $\sfc$ in \cref{eq: c in sym alg} is precisely $c(n; p_1, \ldots, p_m)$, the number of ways to write $\sfc$ as a product of permutations $\pi_1 \cdots \pi_m$ where $\pi_i$ has $n-p_i$ cycles.
\end{example}

\subsubsection{Generating functions for $e_k$ and $h_k$.}
We will use repeatedly the following well-known generating function expressions for the elementary and homogeneous symmetric functions:
\begin{align}
    E(t) &= \sum_{k \geq 0} e_k t^k = \prod_{i=1}^{\infty} (1+x_i t) \label{eq:egenfunction}\\
    H(t) &=  \sum_{k \geq 0} h_k t^k = \prod_{i=1}^{\infty} \frac{1}{1-x_i t}.\label{eq:hgenfunction}
\end{align}
Note that there is a duality between \cref{eq:egenfunction} and \cref{eq:hgenfunction}: namely
\[ E(-t) H(t) = 1 = E(t) H(-t).\]

We will also consider the analogous generating functions for $e_k(\Xi_n)$:
\begin{align}
E_{n}(t):=& \prod_{i=1}^{n} \big( 1 + J_i t\big)  = \sum_{k \geq 0} e_k\big( \Xi_n \big) t^k  \label{eqn:En} \\
\tilde{E}_{n}(t):=&  \prod_{i=1}^{n} \big( t + J_i \big)  = \sum_{k \geq 0} e_k\big( \Xi_n \big) t^{n-k} \label{eqn:Etilden}
\end{align}
By \cref{eq:ekevaluation}, we thus have
\begin{align*}
    E_n(t) =& \sum_{\lambda \vdash n} \cc_\lambda  \, t^{n-\ell(\lambda)}, \\
    \tilde{E}_n(t) =& \sum_{\lambda \vdash n} \cc_\lambda \,  t^{\ell(\lambda)}.
\end{align*}
We will generalize $E_n(t)$ and $\tilde{E}_n(t)$ to the Hecke algebra in \cref{subsec:hecke-algebra}, where they will play an important role in our analysis.

\subsection{The Hecke algebra}
\label{subsec:hecke-algebra}
We now turn to the Hecke algebra $\HH_n(q)$, which deforms $\kk[\symm_n]$ by introducing a parameter $q$. We will often treat $q$ as an (invertible) indeterminate, and we will sometimes specialize it to a specific value in $\kk$ when relevant.
\begin{definition}\label{def:heckealgebra}
    The \emph{type $A$ Iwahori--Hecke algebra $\HH_n(q)$} is the associative $\kk$-algebra generated by elements $T_1, \ldots, T_{n-1}$ subject to the following relations:
    \begin{enumerate}
        \item $(T_i - q)(T_i + 1) = 0$, or equivalently, $T_i^2 = (q-1)T_i + q$ for $1 \leq i \leq n-1$; \smallskip
        \item $T_i T_j = T_j T_i$ for $1 \leq i \neq j \leq n-1$ if $|i-j| \geq 2$; \smallskip
        \item $T_i T_{i+1} T_i = T_{i+1} T_i T_{i+1}$ for $1 \leq i \leq n-2$.
    \end{enumerate}
\end{definition}
When $q\neq 0$ or a root of unity, the algebra $\HH_n(q)$ is semisimple. This is the context that we will be interested in for the remainder of this paper. Note that setting $q=1$ recovers the group algebra $\kk[\symm_n]$, thought of as the associative $\kk$-algebra generated by adjacent transpositions $s_i = (i\,\,i{+}1)$ for $1 \leq i \leq n-1$.

If $w = s_{i_1} s_{i_2} \cdots s_{i_\ell}$ is a reduced expression for $w \in \symm_n$, we define
\[ T_w := T_{i_1} T_{i_2} \cdots T_{i_\ell}.\]
It turns out that $T_w$ is independent of the reduced expression for $w$, and that $\{ T_w : w \in \symm_n \}$ is a linear basis of $\HH_n(q)$, which we will refer to as the \emph{natural basis of $\HH_n(q)$}. 

In what follows, we will see that the character theory of $\HH_n(q)$ is determined by elements $T_w$ for which $w$ is of minimal length in its conjugacy class. We pick a unique such element $w^*_\lambda$ for each $\CC_\lambda$, and write
\begin{equation*}
     T_{\lambda} := T_{w^*_\lambda}.
\end{equation*}
(The choice of $w^*_\lambda$ does not matter, but we make a consistent choice to simplify notation and exposition.)

\begin{remark}
Relation (1) in \cref{def:heckealgebra} is not the only convention for defining $\HH_n(q)$. More generally, one can define the type $A$ Iwahori--Hecke algebra $\HH_n(q_1,q_2)$ subject to the relation
\begin{equation} \label{eq:moregeneralheckedef}
(T_i-q_1)(T_i-q_2) = 0.
\end{equation}
The following specializations of $q_1, q_2$, yield isomorphic algebras:
\begin{align}
    q_1 = q & \quad \quad q_2 = -q^{-1} \label{presentation:symmetrized} \\
    q_1 = q & \quad \quad q_2 = -1\label{presentation:traditional} \\
    q_1 = q^2 & \quad \quad q_2 = -1.\label{presentation:traditionalsquare}
\end{align}
\cref{def:heckealgebra} uses \cref{presentation:traditional}, however \cref{presentation:symmetrized} is sometimes more useful in Kazhdan--Lusztig theory.

In general, the presentation chosen will change the formulas one obtains, but one can always move from one to the other without too much difficulty.
\end{remark}
\subsubsection{Character theory for $\HH_n(q)$}
In this section, we will describe the $\HH_n(q)$ analogues of the symmetric group character theory and properties of $Z(\kk[\symm_n])$ discussed in \cref{section:charactertheory for symmetric group}.

When $\HH_n(q)$ is semisimple, its representation theory bears many important similarities to the symmetric group. For example, the irreducible representations of $\HH_n(q)$ are indexed by integer partitions of $n$, and these modules (also typically denoted by $S^\lambda$) have the same branching properties as those of $\symm_n$; see \cite{mathas1998hecke} for an in-depth discussion. There is a Frobenius characteristic map in this case as well; see \cite{WanWang}. 

However, we will see that even in the semisimple case, there are nuances to the character theory for $\HH_n(q)$ that make it far more complicated than the analogous character theory for $\symm_n$. 

Recall that conjugacy classes of the symmetric group depend only on the cycle type of a permutation and can thus be indexed by partitions.
Moreover, the conjugacy class sums form a natural basis for the center $Z(\kk[\symm_n])$.
In the case of the Hecke algebra, the dimension of $Z(\HH_n(q))$ is also given by the number of partitions $\lambda$ of $n$. However, the analogous sum over $T_w$ for $w \in \CC_\lambda$ is not necessarily central.

Instead, the correct analogue of the conjugacy class sum basis $\{\cc_\mu\}$ is the \emph{Geck--Rouquier basis}
\[ \big \{ \Gamma_{\mu} \big \}_{\mu \vdash n} \]
which is a deformation of $\cc_{\mu}$
in the sense that each $\Gamma_{\mu}|_{q=1} = \cc_{\mu}$. In general, an explicit expression for $[T_w]\Gamma_{\mu}$, the coefficient of $T_w$ in $\Gamma_{\mu}$, is not known. However, when $w$ is of minimal length in its conjugacy class, the coefficient $[T_w]\Gamma_{\mu}$ is well understood, as we explain below.

Geck--Rouquier prove in \cite[\S 5]{geck1997centers} the following characterization of the $\{ \Gamma_\mu \}$, which we shall use as our definition; see also \cite{meliot2010products}.

\begin{theorem}[Geck--Rouquier] \label{thm:GRbasis}
There is a unique family of elements $\{ \Gamma_{\mu} \}_{\mu \vdash n} $ in $\HH_n(q)$ characterized by the following conditions: \smallskip
    \begin{enumerate}
        \item $\Gamma_{\mu}$ is in $Z(\HH_n(q))$ \smallskip
        \item $\Gamma_{\mu}|_{q=1} = \cc_{\mu}$, and  \smallskip
        \item writing $\Gamma_{\mu}$ in the natural basis of $\HH_n(q)$:
        \[ \Gamma_{\mu} = \sum_{w \in \symm_n} a_{\mu,w} \  T_w,\]
        we have that
        \[ a_{\mu,w} = \begin{cases} 1 &
        \textrm{if $w$ is of minimal length in its conjugacy class and has cycle type } \mu \\
        0 &
        \textrm{if $w$ is of minimal length in its conjugacy class and has cycle type } \lambda \neq \mu.\\
        \end{cases}\]
    \end{enumerate}
    In fact, these elements form a basis of $Z(\HH_n(q))$.
\end{theorem}
We will refer to the basis elements $\Gamma_\mu$ as the \emph{Geck--Rouquier basis}.
Note that condition (3) does not impose any conditions on the coefficients $a_{\mu,w}$ when $w$ is not of minimal length in its conjugacy class.
\begin{example}
We show below the Geck--Rouquier basis of $Z(\HH_4(q))$.
Elements $T_w$ are written below with $w$ in cycle notation; the terms in blue correspond to those $T_w$ for which $w$ is of minimal length in the appropriate conjugacy class, and thus their coefficient is $1$.
\begin{align*}
    \Gamma_{(4)} & =
    \begin{multlined}[t]
        (1 - q^{-1})\Big((1-q^{-2})T_{(14)(23)} + (1-q^{-1})T_{(14)} + T_{(142)} + T_{(143)} + T_{(13)(24)} + T_{(134)} + T_{(124)} \Big) + \\
        (1-q^{-1}+q^{-2})T_{(1423)} + (1-q^{-1}+q^{-2})T_{(1324)} + {\color{RoyalBlue}T_{(1432)} + T_{(1342)} + T_{(1243)} + T_{(1234)}},
    \end{multlined} \smallskip
    \\
    \Gamma_{(3,1)} & =
    \begin{multlined}[t]
        (1-q^{-1})\Big( (q^{-1}-q^{-2}) T_{(14)(23)} + q^{-1}T_{(1423)} + q^{-1}T_{(1324)} + 2q^{-1}T_{(14)} + T_{(13)} + T_{(24)} \Big) + \\
        q^{-1}T_{(142)} + q^{-1}T_{(143)} + q^{-1}T_{(134)} + q^{-1}T_{(124)} + {\color{RoyalBlue}T_{(132)} + T_{(123)} + T_{(243)} + T_{(234)}},
    \end{multlined} \smallskip
        \\
    \Gamma_{(2,2)} & = q^{-2}T_{(14)(23)} + q^{-1}T_{(13)(24)} + {\color{RoyalBlue}T_{(12)(34)}},
  \smallskip  \\ 
    \Gamma_{(2,1,1)} & = q^{-2}T_{(14)} + q^{-1}T_{(13)} + q^{-1}T_{(24)} + {\color{RoyalBlue}T_{(12)} + T_{(34)} + T_{(23)}},
 \smallskip   \\
    \Gamma_{(1,1,1,1)} & = {\color{RoyalBlue}T_{\epsilon}}.
\end{align*}
\end{example}

In \cite{ram1991frobenius}, Ram develops the character theory for $\HH_n(q)$ in
the semisimple case. One way of interpreting his results is as a
$q$-deformation of the Murnahan--Nakayama rule for symmetric group characters.
We write $\chi^{\lambda}$ for the character of the irreducible
representation indexed by $\lambda$, and refer the reader to
\cite[Thm~5.4]{ram1991frobenius} for details on this formula.

Similarly to the Geck--Rouquier basis, the only explicit formulas for the characters $\chi^\lambda$ of $\HH_n(q)$ are for elements $T_w$ where $w$ is of minimal length in its conjugacy class. The character values of all other elements $T_w$ in $\HH_n(q)$ are determined by $\chi^{\lambda}(T_{\lambda})$ and can be computed recursively (see \cite[Corollary 5.2]{ram1991frobenius}).

\begin{example}
Consider the character $\chi^\lambda$ for any $\lambda$, evaluated at $T_{\sfc}$. Ram's formulas imply the following:
    \label{ex:qchar-longcycle}
    \[ \chi^\lambda(T_{\sfc}) = \begin{cases}
        (-1)^{k}q^{n-k-1} & \lambda = (n-k,1^k) \\
        0 & \textrm{ otherwise}
    \end{cases}. \]
This computation will be useful in what follows, and in fact hints at why factorization results for $\sfc$ and $T_{\sfc}$ are more tractible than other, arbitrary permutations.
\end{example}

Ram also gives an analogue of \cref{eq:isotypic-projectors} for isotypic projectors $\{ \pi_\lambda \}_{\lambda \vdash n}$ in $\HH_n(q)$ in \cite[Eqn.~2.5]{ram1991frobenius}. His formula involves the dimensions of certain irreducible representations of ${\sf GL}_n(\mathbb{F}_q)$ called unipotent representations (see \cite[\S 4.2, 4.6, 4.7]{grinberg2014hopf} and \cite[\S 3.2]{brauner2023invariant} for more on this class of representations).

In particular, write $f^\lambda(q)$ to be the dimension of the irreducible unipotent representation of ${\sf GL}_n(\mathbb{F}_q)$ indexed by $\lambda$. The following explicit formula for $f^\lambda(q)$ is proved in \cite[Eqn. 4.2]{lewis2014reflection}:
\begin{equation}\label{eq:f^lam}
    f^\lambda(q) = (1-q) (1-q^2) \cdots (1-q^n) \frac{q^{n(\lambda)}}{\prod_{(i,j) \in \lambda} (1-q^{h(i,j)})},
\end{equation}
where $h(i,j)$ is the hook-length of the cell $(i,j)$ in $\lambda$ from \cref{eq:hooklength}, and
\[ n(\lambda) = \sum_{i \geq 1} (i-1) \lambda_i.\]
Equivalently, multiplying and dividing \cref{eq:f^lam} by $(1-q)^n$,
\[
    f^\lambda(q) = \dfrac{q^{n(\lambda)} [n]!_q}{\prod_{(i,j) \in \lambda} [h(i,j)]_q}.
\]

\begin{example}\label{ex:qhookdimension}
    Let $\lambda = (n-k,1^k)$. Then the hook length of the cells of $\lambda$ are given by the multiset
    \begin{equation*}
        \big\{\!\! \big\{ \
        \lefteqn{\overbrace{\phantom{1, \ 2, \ \ldots, \ k, \ n}}^{\text{first column}}}
        1, \ 2, \ \ldots, \ k, \ \underbrace{n, \ n - k - 1, \ n - k - 2, \ \ldots, \ 2, \ 1}_{\text{first row}}
        \ \big\}\!\! \big\}
    \end{equation*}
    where the cells are read from bottom-to-top along the first column
    and then left-to-right along the first row.
    Thus,
    \[ n(\lambda) = 0 (n-k) + 1 + 2 + \cdots + k = \binom{k+1}{2}  \]
    and
    \begin{align*}
        f^{(n-k,1^k)}(q)
        &= \dfrac{q^{\binom{k+1}{2}} [n]!_q}{[1]_q\cdots[k]_q [1]_q \cdots [n-k-1]_q [n]_q} \\
         &= q^{\binom{k+1}{2}} \frac{[n-1]!_q}{[k]!_q [n-k-1]!_q}.
    \end{align*}
\end{example}

Ram proves the following, which gives an analogue of \cref{eq:pi.to.c} for $\HH_n(q)$:
\begin{proposition}[\cite{ram2024lusztig}]\label{prop:projectorformula}
The element $\pi_\lambda$ has the following expression in the natural basis of $\HH_n(q)$:
\begin{equation}\label{eq:equationforp}
    \pi_\lambda = \frac{f^\lambda(q)}{[n]!_q} \sum_{w \in S_n} \chi^\lambda(T_w) q^{-\ell(w)} T_{w^{-1}}.
\end{equation}
Thus, $\pi_\lambda$ can be expressed in the Geck--Rouquier basis of $Z(\HH_n(q))$ as

    \begin{equation} \label{eq:projectorsinGRbasis}
\pi_\lambda = \frac{f^\lambda(q)}{[n]!_q}\sum_{\mu \vdash n} \chi^\lambda(T_{\mu}) q^{-\ell(w^*_\mu)} \Gamma_\mu.\end{equation}
\end{proposition}

Just as in the case of \cref{eq:c.to.pi} in the symmetric group, there is an elegant inverse change of basis from the $\Gamma_\mu$ to the $\pi_\lambda$; while this fact is likely known, we provide a proof for completeness.
\begin{proposition}\label{prop:xintermsofp}
The element $\Gamma_\mu$ has the following expression in the basis $\{\pi_\lambda\}_{\lambda \vdash n}$ of $Z(\HH_n(q))$:
    \[
            \Gamma_\mu =  \sum_\lambda \frac{\chi^\lambda(\Gamma_\mu)}{f^\lambda} \pi_\lambda.
    \]
\end{proposition}

\begin{proof}
The element $\pi_\lambda$ acts as multiplication by $\delta_{\lambda,\eta}$ on the irreducible representation of $\HH_n(q)$ indexed by~$\eta$. Therefore,
\[ \chi^\eta(\pi_\lambda) = \begin{cases}
    \chi^\eta(1) = f^\eta & \lambda = \eta\\
    0 & \lambda \neq \eta.
\end{cases} \]
Write $\Gamma_\mu = \sum_{\lambda \vdash n} c_{\mu,\lambda} \pi_\lambda$ and apply $\chi^\eta$ to both sides of the equality to obtain:
\[
\chi^\eta(\Gamma_\mu) = \sum_{\lambda \vdash n} c_{\mu,\lambda} \chi^\eta(\pi_\lambda) = c_{\mu,\eta} f^\eta.
\]
Solving for $c_{\mu,\eta}$ yields the result.
\end{proof}

\subsubsection{The $q$-Jucys--Murphy elements}
We now turn to the $\HH_n(q)$-analogues of the Jucys--Murphy elements and results in \cref{section.JMforsymmetricgroup}. We first define the \emph{Jucys--Murphy elements of $\HH_n(q)$.}
\begin{definition}
    For any $1 \leq k \leq n$, define $J_k(q) \in \HH_n(q)$ to be the sum
   \begin{equation} \label{eqn.JMforHecke} 
       J_k(q) := q^{1-k} \  T_{(1 \, k)} + q^{2-k} \ T_{(2 \, k)} + \cdots + q^{-1} T_{(k{-}1 \, k)},
    \end{equation}
    where $T_{(i \, k)}$ is the natural basis element indexed by the transposition $(i \; k)$.
    As in the case of the symmetric group, we adopt the convention that $J_1(q) = 0$.
    Observe that $J_k(1) = J_k \in \kk[\symm_n]$.
\end{definition}
\begin{remark}

The elements in \cref{eqn.JMforHecke} are sometimes referred to the \emph{additive Jucys--Murphy elements}, since there are other elements, defined for instance in \cite{elias2017categorical}, that are \emph{multiplicative Jucys--Murphy elements}. These names come from the fact that the $J_k(q)$ in \cref{eqn.JMforHecke} have a joint eigenbasis for $\HH_n(q)$ whose corresponding eigenvalues are an (additive) $q$-analogue of content $\cont_k(T)$, as in \cref{def:qcontent} below. The multiplicative Jucys--Murphy elements also pair-wise commute, and have a shared eigenbasis of $\HH_n(q)$ whose corresponding eigenvalues are a multiplicative $q$-analogue of $\cont_k(T)$, i.e. $q^{\cont_k(T)}$. The latter elements are useful in the context of categorification, but (unlike the $J_k(q)$ above) do not specialize to $J_k$ when $q=1$.
\end{remark}

We will see that the $J_k(q)$ have many of the same remarkable properties as their symmetric group counterparts. First, they also pairwise commute, and thus share a joint eigenbasis of $\HH_n(q)$.

To make this precise, we first define a $q$-analogue of content, which is precisely the $q$-analogue $[\cont_k(T)]_q$.

\begin{definition}\label{def:qcontent}
    For any $T \in \syt(\lambda)$ and $k \in [n]$, the \emph{$q$-content of $T$ at $k$} is
    \[ \qcont_k(T):= \qint{\cont_k(T)} = [\col_k(T) - \row_k(T)]_q.\]
    The \emph{multiset of $q$-contents of $T$} is then
    \[ \qcont(T):= \{\! \{ \qcont_k(T): 1 \leq k \leq n \} \! \}. \]
Once more, for any partition $\lambda$ we set $\qcont(\lambda):= \qcont(T)$, where $T \in \syt(\lambda)$ is any standard Young tableau of shape $\lambda$.
\end{definition}
Note that $\qcont_k(T)$ is not necessarily a polynomial in $q$, since $\cont_k(T)$ can be a negative number; in general it is a Laurent polynomial in $q$.
\begin{example}
Following \cref{ex:contents}, the $q$-contents of $\lambda = (5,4,2,2,1)$ are
\[\qcont\big((5,4,2,2,1)\big) = \big\{\!\!\big\{ \qint{-4},\qint{-3},\qint{-2},\qint{-2},\qint{-1},\qint{-1},\qint{0},\qint{0},\qint{1},\qint{1},\qint{2},\qint{2},\qint{3},\qint{4} \big\}\!\!\big\}.\]
\end{example}
There are seminormal forms $\pi_T \in \HH_n(q)$ satisfying precisely the properties outlined of \cref{lem:interaction-of-Jucys--Murphys-with-pts}, replacing $J_k$ with $J_k(q)$ and $\cont_k(T)$ with $\qcont_k(T)$. The most pertinent of these properties are the following (see \cite{mathas1998hecke}, \cite[Prop. 2.11]{axelrod2024spectrum}, and the references therein):
    \begin{equation}
        \label{eq:isotypic-projectors-properties-hecke}
        \sum_{\lambda \vdash n} \pi_\lambda = 1,
        \qquad
        \sum_{T \in \syt(\lambda)} \pi_T = \pi_\lambda,
        \qquad
        J_k(q) \, \pi_T = \qcont_k(T) \, \pi_T.
    \end{equation}

Finally, the center $Z(\HH_n(q))$ is determined by the $J_k(q)$ in the following sense. Given $f \in \Lambda$, define
\[ f\big(\Xi_n(q)\big):=f \big( J_1(q), J_2(q), \ldots, J_n(q) \big) \in \HH_n(q),\]
so that $f(\Xi_n(1)) = f(\Xi_n)$. We call the map $f \mapsto f(\Xi_n(q))$ the \emph{content evaluation map}.

The importance of the content evaluation map is the following analogue of \cref{thm:isomorphism.center.evaluation}, which was conjectured by Dipper--James in \cite{dipper1987blocks} and proved by Francis--Wang in \cite{francis2006centres}:
\begin{theorem}[Francis--Wang]
    \label{thm:FrancisWang}
    For any $n \geq 1$, the element $z \in \HH_n(q)$ is in the center $Z(\HH_n(q))$ if and only if it can be expressed as
    $ z = f\big(\Xi_n(q)\big) $
    for some $f \in \Lambda$.
    In other words,
    \[Z(\HH_n(q)) = \big\{ f\big(\Xi_n(q)\big) : f \in \Lambda \big\}. \]
\end{theorem}
The content evaluation map plays an important role in the \emph{$q$-Farahat--Higman algebra}, which studies the centers of all Hecke algebras simultaneously.
In particular, in \cite{ryba2023stable} Ryba shows the map can be used to define an isomorphism between the ring of symmetric functions (tensored with a certain coefficient ring) and the $q$-Farahat--Higman algebra.

Once again, some of the most important examples of evaluations $f(\Xi_n(q))$ come from the elementary symmetric functions.
We define the following $q$-deformation of \cref{eqn:Etilden}, which will play a crucial rule in what follows.
\begin{definition}\label{def:etilde}
    Define, for any $n \geq 1$,
    \begin{align*}
    E_{n}(q;t):=& \prod_{i=1}^{n} \big( 1 + J_i(q) t\big)  = \sum_{k \geq 0} e_k\big( \Xi_n(q) \big) t^k   \\
    \tilde{E}_{n}(q;t):=&  \prod_{i=1}^{n} \big( t + J_i(q) \big)  = \sum_{k \geq 0} e_k\big( \Xi_n(q) \big) t^{n-k}
    \end{align*}
\end{definition}

In \cite[Proposition 7.4]{francis2006centres}, Francis--Graham gives a description of $e_k(\Xi_n(q))$, which in turn gives an alternate formulation for $E_{n}(q ; t)$.

\begin{theorem}[Francis--Graham] \label{thm:ryba JM}
For any $n,k \geq 1$, we have
\[
e_k(\Xi_n(q))  = q^{-k} \sum_{\substack{\lambda \vdash n\\ n-\ell(\lambda) = k}} \Gamma_\lambda.
\]
Thus,
\begin{align*}
E_{n}(q;t)  &= \sum_{\lambda \vdash n} \Gamma_{\lambda} \ t^{n-\ell(\lambda)} \ q^{-(n-\ell(\lambda))} \\
\tilde{E}_n(q;t) &= \sum_{\lambda \vdash n} \Gamma_{\lambda} \ t^{\ell(\lambda)} \ q^{-(n-\ell(\lambda))}.
\end{align*}
\end{theorem}

\section{Main results}
\label{sec:main-results}

In this section, we prove our main results.
We begin by highlighting in \cref{ssec:isotypic-projectors-technique} the main technique that will be used in our proofs.
Among other things, this technique provides a uniform and simple method for proving many
known symmetric group factorization results.
We illustrate this by proving Jackson's \cref{thm:introsymmetricsummary}(3);
to the best of our knowledge, this proof is original.
It will also serve as a blueprint for many of our proofs.

We then prove \cref{thm:intro_q_summary_thm}(1) in \cref{ssec:a-q-proof}, \cref{thm:intro_q_summary_thm}(2) in \cref{ssec:q-b-proof}, and \cref{thm:intro_q_summary_thm}(3) in \cref{ssec:q-c-proof}. \cref{ssec:evaluations-other-symmetric-functions}
uses the results in \cref{ssec:isotypic-projectors-technique} and \cref{thm:intro_q_summary_thm} to prove \cref{thm:intro_reciprocity} evaluating other notable symmetric functions.
\subsection{Using isotypic projectors to evaluate symmetric functions at Jucys--Murphy elements}
\label{ssec:isotypic-projectors-technique}

Recall that we are interested in computing the coefficient of
$T_{\sfc}$ in $f(\Xi_n(q))$, the evaluation of a symmetric function $f$ at the
$q$-Jucys--Murphy elements.
We can break this down into two steps, as follows.
\begin{enumerate}
    \item By \cref{thm:FrancisWang}, the element $f(\Xi_n(q))$
belongs to the center of $\HH_n(q)$.
Since the isotypic projectors $\pi_\lambda$ form a basis of the center,
$f(\Xi_n(q))$ can be expressed as a linear
combination of the $\{\pi_\lambda\}$. \smallskip
\item We determine the coefficients of $T_{\sfc}$ in the
individual isotypic projectors $\pi_{\lambda}$,
which combine to give an expression for the coefficient of
$T_{\sfc}$ in $f(\Xi_n(q))$.
\end{enumerate}

We begin by expressing $f(\Xi_n(q))$ as a linear combination of the isotypic projectors.

\begin{lemma}\label{lemma:isoprojtechnique}
    Let $f$ be any symmetric function. Then
    $f(\Xi_n(q))$ is an element of $Z(\HH_n(q))$ and
    \[ f(\Xi_n(q)) = \sum_{\lambda \vdash n} f\big(\qcont(\lambda)\big) \, \pi_\lambda. \]
\end{lemma}

\begin{proof}
    From \cref{thm:FrancisWang}, we have that $f(\Xi_n(q))$ belongs to the center of $\HH_n(q)$.

    Recall from~\cref{eq:isotypic-projectors-properties-hecke}
    that the elements $\pi_T$ satisfy
    \begin{equation*}
        \sum_{\lambda \vdash n} \pi_\lambda = 1,
        \qquad
        \sum_{T \in \syt(\lambda)} \pi_T = \pi_\lambda,
        \qquad
        J_k(q) \, \pi_T = \qcont_k(T) \, \pi_T.
    \end{equation*}
    Repeatedly applying the last property, we have
    \begin{equation*}
        J_1(q)^{a_1} \cdots J_n(q)^{a_n} \, \pi_T
        = \qcont_1(T)^{a_1} \cdots \qcont_n(T)^{a_n} \, \pi_T.
    \end{equation*}
    Writing $f = \sum_{\alpha \vDash n} f_\alpha x^\alpha$,
    where $f_\alpha$ is the coefficient of the monomial $x^\alpha$,
    it follows that
    \begin{align*}
        f\big(\Xi_n(q)\big) \, \pi_{T}
        & = \bigg(\sum_{\alpha \vDash n} f_\alpha \, x_1^{\alpha_1} x_2^{\alpha_2} \cdots x_n^{\alpha_n}\bigg) \bigg|_{x_k = J_k(q)} \pi_{T} \\
        & = \sum_{\alpha \vDash n} f_\alpha \, J_1(q)^{\alpha_1} J_2(q)^{\alpha_2} \cdots J_n(q)^{\alpha_n} \, \pi_{T} \\
        & = \sum_{\alpha \vDash n} f_\alpha \, \qcont_1(T)^{\alpha_1} \qcont_2(T)^{\alpha_2} \cdots \qcont_n(T)^{\alpha_n} \, \pi_{T} \\
        & = \bigg(\sum_{\alpha \vDash n} f_\alpha \, x_1^{\alpha_1} x_2^{\alpha_2} \cdots x_n^{\alpha_n}\bigg) \bigg|_{x_k = \qcont_k(T)} \pi_{T} \\
        & = f\big(\qcont_1(T), \qcont_2(T), \ldots, \qcont_n(T)\big) \, \pi_{T} \\
        & = f\big(\qcont(\sh(T))\big) \, \pi_{T},
    \end{align*}
    where the last equality follows from the observation that $\qcont(T)$ only
    depends on $\sh(T)$.
    Consequently,
    \begin{equation*}
        \begin{aligned}[b]
            f\big(\Xi_n(q)\big)
            & = \sum_{\lambda \vdash n} f\big(\Xi_n(q)\big) \, \pi_\lambda
            & \text{(because $\textstyle 1 = \sum_{\lambda \vdash n} \pi_\lambda$)}
            \\
            & = \sum_{\lambda \vdash n} \sum_{T \in \syt(\lambda)} f\big(\Xi_n(q)\big) \, \pi_T
            & \text{($\textstyle \pi_\lambda = \sum_{T \in \syt(\lambda)} \pi_T$)}
            \\
            & = \sum_{\lambda \vdash n} \sum_{T \in \syt(\lambda)} f\big(\qcont(\lambda)\big) \, \pi_T
            & \text{\qquad ($\textstyle f(\Xi_n(q)) \pi_{T} = f(\qcont(\lambda)) \pi_{T}$)}
            \\
            & = \sum_{\lambda \vdash n} f\big(\qcont(\lambda)\big) \sum_{T \in \syt(\lambda)} \pi_T
            \\
            & = \sum_{\lambda \vdash n} f\big(\qcont(\lambda)\big) \, \pi_\lambda
            & \text{($\textstyle \sum_{T \in \syt(\lambda)} \pi_T = \pi_\lambda$)}
        \end{aligned}
        \qedhere
    \end{equation*}
\end{proof}

\cref{lemma:isoprojtechnique} expresses $f(\Xi_n(q))$ as a
linear combination of the isotypic projectors $\pi_\lambda$,
reducing the computation of the coefficient of $T_{\sfc}$ in
$f(\Xi_n(q))$ to the computation of its coefficient in each $\pi_\lambda$.

In order to compute the coefficient of $T_{\sfc}$ in $\pi_\lambda$, we need the following result on evaluating
the irreducible characters of the Hecke algebra.

\begin{lemma}\label{lemma:character_minlength}
    Let $\chi^\lambda$ be the character of the irreducible representation
    of $\HH_n(q)$ indexed by $\lambda \vdash n$.
    If $w \in \symm_n$ is of minimal length in its conjugacy class, then
    \[ \chi^\lambda(T_w) = \chi^\lambda(T_{w^{-1}}). \]
\end{lemma}

\begin{proof}
Geck and Pfeiffer show in \cite[Theorem 1.1(b)]{GePf93irrchar} that if $w,w' \in \symm_n$ are both of minimal length in their conjugacy class, then $T_{w}$ and $T_{w'}$ are conjugate in $\HH_n(q)$.
In particular, $\chi(T_w) = \chi(T_{w'})$ for any character $\chi$ of $\HH_n(q)$. In $\symm_n$, any element is conjugate to its inverse and $\ell(w) = \ell(w^{-1})$. The claim thus follows. 
\end{proof}

We can now describe the coefficient of $T_{\sfc}$ in $\pi_\lambda$.

\begin{lemma}\label{lemma:coefficient_longcycle_isotypic}
For $\lambda \vdash n$, the coefficient of $T_{\sfc}$ in the isotypic projector $\pi_{\lambda}$ is
\begin{align*}
[T_{\sfc}] \, \pi_{\lambda}
&=\begin{cases}
    \displaystyle \frac{ (-1)^k q^{\binom{k}{2}}}{\qfact{n}} \qbinom{n-1}{k}{q} &\text{ if } \lambda =(n-k,1^k)\\[1ex]
    0& \text{otherwise}.
\end{cases}
\end{align*}
\end{lemma}

\begin{proof}
Note that the length of the long cycle $\sfc$ is $n-1$,
and by \cref{ex:qhookdimension}
\[ f^{(n-k,1^k)}(q) = \frac{q^{\binom{k+1}{2}} \qfact{n-1}}{\ \qfact{k} \ \qfact{n-k-1}}.  \]
Since $\sfc$ is of minimal length in the conjugacy class indexed by the partition $(n)$,
we have that $\chi^\lambda(T_w) = \chi^\lambda(T_{w^{-1}})$
by \cref{lemma:character_minlength}.
By \cref{prop:projectorformula},
we have $[T_{\sfc}] \pi_\lambda = 0$ if $\lambda \neq (n - k, 1^k)$;
and when $\lambda = (n - k, 1^k)$,
\begin{align*}
      [T_{\sfc}] \pi_{(n-k,1^k)}
      &= \tfrac{f^\lambda(q)}{\qfact{n}} \ \chi^\lambda(T_{\sfc}) \ q^{-(n-1)} \\
      & = \frac{q^{\binom{k+1}{2}} }{\qint{n} \ \qfact{k} \ \qfact{n-k-1}} \ \chi^{(n-k,1^k)}(T_{\sfc}) \ q^{-(n-1)} \\[1ex]
      &=\frac{q^{\binom{k+1}{2}} \cdot (-1)^k q^{n-k-1} \cdot q^{-n+1} }{\qint{n} \ \qfact{k} \ \qfact{n-k-1}}
      & (\text{\cref{ex:qchar-longcycle}}) \\[1ex]
      &= \frac{(-1)^k q^{\binom{k+1}{2} + n-k-1 -n +1}  }{\qint{n} \ \qfact{k} \ \qfact{n-k-1}} \\[1ex]
      &= \frac{(-1)^k q^{\binom{k}{2}} }{\qfact{n}} \qbinom{n-1}{k}{q}.
      &\qedhere
\end{align*}
\end{proof}

The above combine to produce the following expression for
the coefficient of $T_{\sfc}$ in $f(\Xi_n(q))$.

\begin{proposition}
    \label{prop:isoprojtechnique}
    Let $f$ be any symmetric function.
    Then the coefficient of $T_{\sfc}$ in $f(\Xi_n(q))$ is
    \begin{align*}
        [T_{\sfc}] \, f(\Xi_n(q))
        & = \frac{1}{\qfact{n}} \ \sum_{k = 0}^{n - 1} (-1)^k q^{\binom{k}{2}} \qbinom{n-1}{k}{q} f\big(\qcont((n - k, 1^k))\big).
    \end{align*}
\end{proposition}

\begin{proof}
    This follows by combining \cref{lemma:isoprojtechnique,lemma:coefficient_longcycle_isotypic}:
    \begin{equation*}
        \begin{aligned}[b]
            [T_{\sfc}] \, f(\Xi_n(q))
            & = \sum_{\lambda \vdash n} f(\qcont(\lambda)) \, [T_{\sfc}] \, \pi_\lambda \\
            & = \sum_{k = 0}^{n - 1} f(\qcont((n - k, 1^k)) \, [T_{\sfc}] \, \pi_{(n - k, 1^k)} \\
            & = \frac{1}{\qfact{n}} \ \sum_{k = 0}^{n - 1} (-1)^k q^{\binom{k}{2}} \qbinom{n-1}{k}{q} f\big(\qcont((n - k, 1^k))\big).
        \end{aligned}
        \qedhere
    \end{equation*}
\end{proof}

\subsubsection{Prototypical example for the symmetric group}
\label{sec:prototypical-example-isotypic-projectors-technique}

We demonstrate how \cref{prop:isoprojtechnique} can be used to prove
\cref{thm:introsymmetricsummary}(3) in the case of the symmetric group ($q = 1$).
To the best of our knowledge, this proof is original,
and it serves as a blueprint for our proof
of its $q$-analogue in the Hecke algebra (\cref{thm:HeckeJacksonMs}).

\begin{proof}[Proof of \cref{thm:introsymmetricsummary}(3)]
    Let $F_n({\bf t})$ be the right-hand-side of \cref{eq:JacksonSn},
    where ${\bf t} = (t_1, t_2, \ldots, t_m)$.
    Note that we can express $F_n({\bf t})$ as the coefficient
    of $\sfc$ of a generating function with coefficients
    in $\kk[\symm_n]$:
    \begin{equation} \label{eq: F in terms of Es}
        F_n({\bf t}) = [\sfc] \ \tilde{E}_n(t_1) \, \tilde{E}_n(t_2) \, \cdots \, \tilde{E}_n(t_m),
    \end{equation}
    where $\tilde{E}_n$ is as in \cref{eqn:Etilden}; that is,
    \begin{equation}
        \label{eq:En-defn}
        \tilde{E}_n(t_j) = (t_j + J_1) (t_j + J_2) \cdots (t_j + J_n)
        = \sum_{i=0}^{n} e_i(\Xi_n) \, t_j^{n-i}.
    \end{equation}
    Combining \cref{eq: F in terms of Es,eq:En-defn}, we have
    \begin{equation}
        \label{eq:blueprint-F-expanded}
        \begin{aligned}[b]
            F_n({\bf t})
            & =
            [\sfc]
            \Big( \tilde{E}_n(t_1) \ \tilde{E}_n(t_2) \ \cdots \ \tilde{E}_n(t_n) \Big)
            \\ & =
            [\sfc]
            \Big(
                \sum_{\alpha_1 = 0}^{n} e_{\alpha_1}(\Xi_n) t_1^{n - \alpha_1}
                \sum_{\alpha_2 = 0}^{n} e_{\alpha_2}(\Xi_n) t_2^{n - \alpha_2}
                \cdots
                \sum_{\alpha_m = 0}^{n} e_{\alpha_m}(\Xi_n) t_m^{n - \alpha_m}
            \Big)
            \\ & =
            \sum_{0 \leq \alpha_1, \alpha_2, \ldots, \alpha_m \leq n}
            \underbrace{[\sfc] \, e_{\alpha_1}(\Xi_n) \cdots e_{\alpha_m}(\Xi_n)}_{[\sfc] \, e_\alpha(\Xi_n)} \, t_1^{n - \alpha_1} \cdots t_m^{n - \alpha_m}
        \end{aligned}
    \end{equation}
    Applying \cref{prop:isoprojtechnique} to $[\sfc] \, e_{\alpha}(\Xi_n)$ yeilds,
    \begin{equation}
        \label{eq:blueprint-F-using-tilde-E}
        \begin{aligned}[b]
            F_n({\bf t})
            & =
            \sum_{\alpha} \Bigg(\frac{1}{n!} \sum_{k=0}^{n-1} (-1)^k \binom{n-1}{k} e_\alpha\big( \cont((n{-}k,1^k)) \big)\Bigg) \, t_1^{n - \alpha_1} \cdots t_m^{n - \alpha_m}
            \\ & =
            \frac{1}{n!} \sum_{k=0}^{n-1} (-1)^k \binom{n-1}{k}
            \prod_{j=1}^{m}
            \Bigg(\sum_{\alpha_j = 0}^{n} e_{\alpha_j}\big(\cont((n{-}k,1^k))\big) t_j^{n - \alpha_j}\Bigg)
            \\ & =
            \frac{1}{n!} \sum_{k=0}^{n-1} (-1)^k \binom{n-1}{k} \tilde{E}_{n,k}(t_1) \cdots \tilde{E}_{n,k}(t_m),
        \end{aligned}
    \end{equation}
    where
    \begin{equation}
        \tilde{E}_{n, k}(t) \ := \ \sum_{j = 0}^{n} e_{j}\big(\cont((n{-}k,1^k))\big) t^{n - j}
        \ = \quad \smashoperator{\prod_{c \in \cont((n{-}k, 1^k))}} \quad \big(t_{j} + c \big).
    \end{equation}

    Next, we compute $\tilde{E}_{n, k}(t)$.
    We have that
    \begin{equation*}
        \cont\big((n{-}k,1^k)\big) = \big\{\! \big\{
            -k, -k + 1, \ldots, -k + (n - 2), -k + (n - 1)
        \big\}\! \big\}
    \end{equation*}
    since the coordinates of the cells in $(n{-}k, 1^k)$
    are (reading bottom-to-top along the first column
    and then left-to-right along the first row):
    \begin{equation*}
        \lefteqn{\overbrace{\phantom{(k{+}1, 1), \ (k, 1), \ \ldots, \ (2, 1), \ (1, 1)}}^{\text{first column}}}
        (k{+}1, 1), \ (k, 1), \ \ldots, \ (2, 1), \ \underbrace{(1, 1), \ (1, 2), \ \ldots, \ (1, n{-}k)}_{\text{first row}}.
    \end{equation*}
    Hence,
    \begin{equation}
        \label{eq:tildeE-symm}
        \tilde{E}_{n, k}(t)
        = \prod_{c \in \cont((n{-}k, 1^k))} \big(t_{j} + c \big)
        = (t_{j} - k)^{(n)},
    \end{equation}
    where $(t_{j} - k)^{(n)}$ denotes the \emph{rising factorial}
    \begin{equation*}
        (t - k)^{(n)} := (t - k) (t - k + 1) \cdots (t + (n - k - 1)).
    \end{equation*}
    Now we perform a change of basis,
    expressing each rising factorial $(t - k)^{(n)}$ in
    the \emph{binomial basis} $\binom{t}{n - r}$.
    Explicitly, by the \emph{Chu--Vandermonde identity} we have that
    \begin{equation}
        \label{eq:chu-vandermonde}
        (t-k)^{(n)}
        = n!\sum_{r=k+1}^n \binom{n-1-k}{r-1-k}\binom{t}{r}
        = n!\sum_{r = 0}^{n - k - 1} \binom{n-1-k}{r}\binom{t}{n - r}.
    \end{equation}

    Therefore,
    \begin{align*}
        F_n({\bf t})
        & =
        \frac{1}{n!} \sum_{k=0}^{n-1} (-1)^k \binom{n-1}{k} \, (t_1 - k)^{(n)} \, (t_2 - k)^{(n)} \, \cdots \, (t_m - k)^{(n)}
        \\ & =
        n!^{m-1}
        \sum_{k=0}^{n-1}
        \sum_{0 \leq r_1,\ldots,r_m \leq n-k-1}
        (-1)^k \binom{n-1}{k} \left(\prod_{i = 1}^{m} \binom{n-1-k}{r_i}\right)
        \binom{t_1}{n - r_1} \cdots \binom{t_m}{n - r_m}
        \\ & =
        n!^{m-1}
        \sum_{0 \leq r_1,\ldots,r_m \leq n-1}
        \left(
        \sum_{k=0}^{n-1 - \max_i\{r_i\}}
        (-1)^k \binom{n-1}{k} \prod_{i = 1}^{m} \binom{n-1-k}{r_i}
        \right)
        \binom{t_1}{n - r_1} \cdots \binom{t_m}{n - r_m}.
    \end{align*}
    Using an inclusion-exclusion argument (e.g. see \S4.2.2 or \cite[Lemma 2.2.23]{moralesPhd}),
    the alternating sum in parentheses above can be shown to be
    \begin{equation} \label{eq:rel Ms}
        \sum_{k=0}^{n-1 - \max_i\{r_i\}}
        (-1)^k \binom{n-1}{k} \prod_{i = 1}^{m} \binom{n-1-k}{r_i}
        =
        \M^{n-1}_{(r_1, r_2, \ldots, r_m)}.
    \end{equation}
    Thus,
    \begin{equation*}
        F_n({\bf t})
        = n!^{m-1} \sum_{0 \leq r_1, \ldots, r_m \leq n - 1} \M^{n - 1}_{(r_1, \ldots, r_m)}
        \binom{t_1}{n - r_1} \cdots \binom{t_m}{n - r_m}.
        \qedhere
    \end{equation*}
\end{proof}

\begin{remark}
The coefficients $\M_{\bf r}^n$ satisfy several natural cancellation-free recurrences. For example, consider whether each set $S_i$ in the tuple $(S_1,\ldots,S_m)$ contains the element $n$; we see that  \begin{equation} \label{eq:key recurrence M}
\M^n_{\bf r} \,=\, \sum_{\varnothing \neq T\subseteq [m]} \M^{n-1}_{{\bf r} - {\bf e}_T},
\end{equation}
where ${\bf e}_T$ is the indicator vector of the set $T$. We will prove a $q$-analogue of \cref{eq:key recurrence M} in \cref{prop:recursion-qMs}.

Alternatively, by 
\begin{enumerate}
    \item first picking $(S_1,\ldots,S_{m-1})$ (there are $\binom{n}{i}\cdot \M^{i}_{(r_1,\ldots,r_{m-1})}$ choices if the cardinality of $T_{m-1} := S_1 \cup \cdots \cup S_{m-1}$ is $i$),
    \item and then choosing $S_m$ (fully determined by its intersection with $T_{m-1}$, which must have cardinality $r_m-n+i$),
\end{enumerate}
we obtain the following recursion:
\begin{equation}
\label{eq:Bolan recurrence}
\M^n_{(r_1,\ldots,r_m)} = \sum_{i=0}^n \binom{n}{i}\binom{i}{r_m-n+1} \M^i_{(r_1,\ldots,r_{m-1})}.
\end{equation}
We give a $q$-analogue of \cref{eq:Bolan recurrence} in \cref{prop:positive recurrence Ms}, originally due to Matthew Bolan.
\end{remark}

\begin{remark} \label{rem:gf Ms}
The coefficients $\M_{\bf r}^n$ also have the following generating polynomial:
\begin{equation} \label{eq:gf for Ms}
\sum_{0\leq r_1,\ldots, r_m\leq n} \M^n_{\bf r} x_1^{n-r_1}\cdots x_m^{n-r_m} = \Bigl( (1+x_1)\cdots (1+x_m)-x_1\cdots x_m  \Bigr)^n.
\end{equation}
In fact, this is how these coefficients were first implicitly defined by Jackson in \cite[Thm. 4.3]{Jackson0}. This identity can be obtained using the recurrence \cref{eq:key recurrence M}.
\end{remark}

\subsection{Proof of \cref{thm:intro_q_summary_thm}(1): $a_q(n;j)$ generating function}\label{ssec:a-q-proof}

We will now prove \cref{thm:intro_q_summary_thm}(1); this is a $q$-analogue of \cref{thm:introsymmetricsummary}(1), which
describes the generating function for $a(n;j)$, the number of factorizations of $\sfc$
into $j$ transpositions.

Recall that we define
\begin{align}
    a_q(n;k) :=& [T_{\sfc}] \,\, {\color{RoyalBlue} e_{1^k}(\Xi_n(q))}  \\
    =& [T_{\sfc}] \,\,
    {\color{RoyalBlue}
        \Big( \hskip3pt \smashoperator{\sum_{1\leq i<j\leq n}} \ q^{i-j} T_{(i\,j)} \hskip3pt \Big)^k
    }. \label{eq:a-in-hecke}
\end{align}
Note that \cref{eq:a-in-hecke} is obtained by evaluating the polynomial $e_1^k(x_1, \ldots, x_n)$ at $J_1(q), \ldots, J_n(q)$.

\begin{example} \label{ex:def aq}
For $n=3$ and $k=2$ we have that 
\begin{align*}
a_q(3,2) 
&= [T_{\sfc[3]}]\,  \big( J_2(q) + J_3(q) \big)^2 \\
&= [T_{\sfc[3]}]\, \left( q^{-1}T_{(1,2)} + q^{-2}T_{(1,3)}+q^{-1}T_{(2,3)}\right)^2\\
&= [T_{\sfc[3]}]\,\Big(q^{-2} T_{(1,2)}^2 + q^{-2} T_{(1,2)}T_{(2,3)} + q^{-3}T_{(1,2)}T_{(1,3)} + q^{-2}T_{(2,3)}T_{(1,2)} + q^{-2}T_{(2,3)}^2 + \\
& \hspace*{.3\linewidth} q^{-3}T_{(2,3)}T_{(1,3)} + q^{-3}T_{(1,3)}T_{(1,2)} + q^{-3}T_{(1,3)}T_{(2,3)} + q^{-4} T_{(1,3)}^2 \Big)\\
&= 0 + q^{-2} + 0 + 0 +0+ q^{-2}+q^{-2} + q^{-2} + (q^3-2q^{-2}+q^{-1})\\
&= q^{-3}+q^{-2}+q^{-1}\\
&= q^{-3}\ [3]_q 
\end{align*}
\end{example}

We will use the following lemma to compute the generating function for $a_q(n;j)$ in \cref{thm:tree conj} below.
\begin{lemma}\label{lemma:q-form-prodfac}
    For all integers $0 \leq k < n$,
    \begin{equation}
    \label{eq:q-form-prodfac}
        \qfact{n-1-k} \qfact{k}
        = \dfrac{(-1)^k q^{\binom{n}{2} + \binom{k}{2}}}{(1-q)^{n-1}} \prod_{ \substack{0 \leq j \leq n-1 \\ j \neq k} } (q^{-j} - q^{-k}).
    \end{equation}
\end{lemma}

\begin{proof}
    For $1 \leq i \leq n-1-k$, write $1 - q^i = q^{k+i} (q^{-(k+i) - q^{-k})}$ to obtain
    \[
        \qfact{n-1-k} = \frac{q^{\sum_{j = k+1}^{n-1} j}}{(1-q)^{n-1-k}} \prod_{k < j \leq n-1} (q^{-j} - q^{-k}).
    \]
    Similarly, for $1 \leq i \leq k$ write $1 - q^i = -q^{k} (q^{-(k-i) - q^{-k})}$ to obtain
    \[
        \qfact{k} = \frac{(-1)^k q^{k^2}}{(1-q)^k} \prod_{0 \leq j < k} (q^{-j} - q^{-k}).
    \]
    The result follows since
    \[ k^2 = \sum_{j = 1}^k j + \sum_{j = 1}^{k-1} j \]
    and therefore
    \[
     k^2 + \sum_{j = k+1}^{n-1} j   = \binom{n}{2} + \binom{k}{2}. \qedhere \]
\end{proof}

We now give a closed formula for the generating function of the numbers $a_q(n;j)$.

\begin{theorem}
\label{thm:tree conj}
    For a fixed $n \in \Z_{\geq 0}$, the ordinary generating function of the numbers $\{a_q(n;j)\}_j$ is
    \begin{equation}
        \label{eq:q gen fcn for e's}
         \sum_{j \geq 0} a_q(n;j) t^j
      = \dfrac{q^{-\binom{n}{2}} \qint{n}^{n-2} t^{n-1}}{\prod_{j=0}^{n-1} \left( 1 - \left( \tfrac{n - q^{-j}\qint{n}}{1-q} \right)t \right)}.
    \end{equation}
\end{theorem}

\begin{proof}
We will make use of \cref{prop:isoprojtechnique},
which requires evaluating the symmetric function $e_1$
at the $q$-content of the hook partition $(n - k, 1^k)$.
Since
\[
    \qcont((n-k, 1^k)) = \big\{\!\!\big\{ \qint{j} : -k \leq j \leq n - k - 1 \big\}\!\!\big\},
\]
we have
\[
    e_1\big(\qcont((n-k,1^k))\big)
    = \sum_{j = -k}^{n-k-1} \qint{j}
    =  \frac{\sum_{j = -k}^{n-k-1}(1-q^j)}{1-q}
    = \frac{n-q^{-k}\qint{n}}{1-q}.
\]
We apply \cref{prop:isoprojtechnique} to the geometric series
\[
    \sum_{j\geq 0} e_1(x_1,\ldots,x_n)^j t^j = \frac{1}{1-e_1(x_1,\ldots,x_n) t}
\]
to obtain
\[
    \sum_{j \geq 0} a_q(n;j) t^j
    = \frac{1}{\qfact{n}} \sum_{k=0}^{n-1} (-1)^kq^{\binom{k}{2}} \qbinom{n-1}{k}{q} \cdot \frac{1}{1-\left(\frac{n - q^{-k}\qint{n}}{1-q}\right)t}.
\]
We would thus like to prove that
\[
    \frac{1}{\qfact{n}}
    \sum_{k=0}^{n-1} (-1)^kq^{\binom{k}{2}} \qbinom{n-1}{k}{q} \cdot \frac{1}{1-\left(\frac{n - q^{-k}\qint{n}}{1-q}\right)t}
    = \dfrac{q^{-\binom{n}{2}} \qint{n}^{n-2} t^{n-1}}{\prod_{j=0}^{n-1} \left( 1 - \left( \tfrac{n - q^{-j}\qint{n}}{1-q} \right)t \right)}.
\]

After multiplying both sides by
\[ \prod_{j=0}^{n-1} \left( 1 - \left( \frac{n - q^{-j}\qint{n}}{1-q} \right)t \right), \]
this is equivalent to proving the following equality:
\begin{equation}\label{eq:restating-a-equality}
    \frac{1}{\qfact{n}}
    \sum_{k=0}^{n-1} \Bigg( (-1)^k q^{\binom{k}{2}} \qbinom{n-1}{k}{q} \cdot \prod_{ \substack{0 \leq j \leq n-1 \\ j \neq k} } \left( 1 - \left( \frac{n - q^{-j}\qint{n}}{1-q} \right)t \right) \Bigg) = q^{-\binom{n}{2}} \qint{n}^{n-2} t^{n-1}.
\end{equation}
Since \cref{eq:restating-a-equality} is an equality of polynomials in $t$ of degree $n-1$, it suffices to verify it for $n$ different values of $t$. Substituting
\[ t = \frac{1-q}{n - q^{-k}\qint{n}} \] for $k = 0,1,\ldots,n-1$ in the left-hand-side, all but one of the summands become $0$, and the expression reduces to
\begin{align*}
    & \frac{(-1)^k q^{\binom{k}{2}}}{\qfact{n}} \qbinom{n-1}{k}{q} \cdot \prod_{ \substack{0 \leq j \leq n-1 \\ j \neq k} } \left( 1 - \left( \frac{n - q^{-j}\qint{n}}{n - q^{-k}\qint{n}} \right) \right)
    \\
    & = \frac{(-1)^k q^{\binom{k}{2}}}{\qint{n}\qfact{k}\qfact{n-k-1}} \prod_{ \substack{0 \leq j \leq n-1 \\ j \neq k} } \left( \frac{(q^{-j} - q^{-k})\qint{n}}{n - q^{-k}\qint{n}} \right) \\
    & = q^{-\binom{n}{2}} \qint{n}^{n-2} \left( \frac{1-q}{n - q^{-k}\qint{n}} \right)^{n-1},
\end{align*}
which is exactly the right-hand-side of \cref{eq:restating-a-equality} evaluated at $t = \tfrac{1-q}{n - q^{-k}\qint{n}}$. The last equality above follows from \cref{lemma:q-form-prodfac}.
\end{proof}

Extracting the coefficient of the lowest-degree nonzero term in \cref{thm:tree conj}, we deduce the following result (see \cref{ex:def aq}).

\begin{corollary}
    For all $n \in \Z_{\geq 0}$,
    \begin{equation}
        \label{eq: q identity e's}
        a_q(n;n-1)
        = q^{\binom{n - 2}{2} - 1}(\qint[q^{-1}]{n})^{n-2} = q^{-\binom{n}{2}} \qint{n}^{n-2}.
    \end{equation}
\end{corollary}

More generally, since the generating function of the homogeneous symmetric
functions $h_k$ satisfies
\[
    \sum_{k \geq 0} h_k(x_1, x_2, \ldots) t^k = \prod_{i=1}^{\infty} \frac{1}{1-x_i t}
\]
we can rewrite the denominator on the right-hand-side of~\cref{eq:q gen fcn for e's}
to obtain the following formula.

\begin{corollary}
    For all $n$ and $m$ in $\Z_{\geq 0}$, we have
    \[
        a_q(n;n+m-1) =
        q^{-\binom{n}{2}} \qint{n}^{n-2} h_m\left(\frac{n-\qint{n}}{1-q},\frac{n-q^{-1}\qint{n}}{1-q},\ldots,\frac{n-q^{-n+1}\qint{n}}{1-q}\right).
    \]
\end{corollary}

\begin{proof}
    By \cref{thm:tree conj} and \cref{eq:hgenfunction},
    \begin{align*}
        \sum_{j \geq 0} a_q(n;j) t^j
        & =
        q^{-\binom{n}{2}} \qint{n}^{n-2} t^{n-1} \dfrac{1}{\prod_{j=0}^{n-1} \left( 1 - \left( \tfrac{n - q^{-j}\qint{n}}{1-q} \right)t \right)}
        \\ & =
        q^{-\binom{n}{2}} \qint{n}^{n-2} \sum_{k \geq 0} h_{k}
        \left(\frac{n-\qint{n}}{1-q},\frac{n-q^{-1}\qint{n}}{1-q},\ldots,\frac{n-q^{-n+1}\qint{n}}{1-q}\right) t^{k + n - 1}.
    \end{align*}
    Extracting the coefficient of $t^{n + m - 1}$ on both sides, we have
    \begin{align*}
        a_q(n;n+m-1)
        & =
        q^{-\binom{n}{2}} \qint{n}^{n-2} h_{m}
        \left(\frac{n-\qint{n}}{1-q},\frac{n-q^{-1}\qint{n}}{1-q},\ldots,\frac{n-q^{-n+1}\qint{n}}{1-q}\right).
        \qedhere
    \end{align*}
\end{proof}

\subsection{Proof of \cref{thm:intro_q_summary_thm}(2):  generating function for $b_q(n;j)$}\label{ssec:q-b-proof}

We now prove \cref{thm:intro_q_summary_thm}(2), which gives a $q$-analogue of \cref{thm:introsymmetricsummary}(2) describing the generating function for $b(n;j)$, the number of \emph{monotone} factorizations of $\sfc$
into $j$ transpositions.

Recall that we define
\begin{align}
    b_q(n;j) :=& \, \, [T_{\sfc}] \,\, {\color{RoyalBlue} h_{j}(\Xi_n(q))} \smallskip \\
    =& \, \, [T_{\sfc}] \,\,
    {\color{RoyalBlue}
        \hskip13pt
        \smashoperator{\sum_{\substack{
                    2\leq t_1\leq\cdots \leq t_j\leq n \\
                    s_1<t_1,\ldots,s_j<t_j
            }}
        }
        \hskip5pt
        q^{\sum_{i} (s_i - t_i)} T_{(s_1\,t_1)} \cdots T_{(s_j\,t_j)}. \label{eq:b in hecke}
    }
\end{align}
\cref{eq:b in hecke} comes from plugging $J_1(q), \ldots, J_n(q)$ into the symmetric function $h_j$.

\begin{example} \label{ex:def bq}
For $n=3$ and $j=2$, we have that 
\begin{align*}
b_q(3,2) &= [T_{\sfc[3]}] \big( J_2(q)^2 + J_2(q)J_3(q) + J_3(q)^2 \big)\\
&= [T_{\sfc[3]}] \Big(q^{-2} T_{(1,2)}^2 + q^{-2} T_{(1,2)}T_{(2,3)} + q^{-3}T_{(1,2)}T_{(1,3)} + q^{-2}T_{(2,3)}T_{(1,2)} + q^{-2}T_{(2,3)}^2 + \\
& \hspace*{.3\linewidth} q^{-3}T_{(2,3)}T_{(1,3)} + q^{-3}T_{(1,3)}T_{(1,2)} + q^{-3}T_{(1,3)}T_{(2,3)} + q^{-4} T_{(1,3)}^2 \Big)\\
&= 0 + q^{-2} + 0 +0+ q^{-2} + q^{-2} + (q^3-2q^{-2}+q^{-1}) \\
&= q^{-3}+q^{-1} = q^{-3} C_2(q).
\end{align*}
\end{example}

To provide a closed formula for the generating function for $b_q(n;j)$, we will use the following $q$-analogue of the Vandermonde identity.

\begin{lemma}[$q$-Chu--Vandermonde identity]
    \label{lemma:qvandermonde}
    For nonnegative integers $a,b,c$,
    \begin{align}
        \qbinom{a + b}{c}{q}
        &= \sum_k q^{k(k+b-c)} \qbinom{a}{k}{q} \qbinom{b}{c-k}{q}  \label{eq:qVand1}\\
        &= \sum_j q^{(c-j)(b-j)} \qbinom{b}{j}{q} \qbinom{a}{c-j}{q}. \label{eq:qVand2}
    \end{align}
\end{lemma}

\begin{proof}
The first formula is \cite[Solution to exercise 1.100]{EC1}.
The second is obtained by the substitution $k = c - j$.
\end{proof}
We now give a generating function for $b_q(n;j)$.
\begin{theorem}
    \label{thm:cat conj}
   For any fixed $n \in \Z_{\geq 0}$,
    the ordinary generating function of the numbers $\{b_q(n;j)\}_j$ is
    \begin{equation}
    \label{eq:thm-cat-conj}
       \sum_{j \geq 0} b_q(n;j) t^j =
       \dfrac{q^{-\binom{n}{2}} C_{n-1}(q) t^{n-1}}{\prod_{j=-(n-1)}^{n-1} (1 - \qint{j}t)},
    \end{equation}
    where $C_n(q) = \dfrac{1}{\qint{n+1}} \qbinom{2n}{n}{q}$ is a $q$-analogue of the Catalan numbers.
\end{theorem}

\begin{proof}
As in the proof of \cref{thm:tree conj}, we apply \cref{prop:isoprojtechnique} to the power series
\[
    \sum_{j \geq 0} h_j(x_1,\ldots,x_n) t^j = \prod_{j = 1}^n \frac{1}{1-x_i t}
\]
to obtain
\begin{align}
    \sum_{j \geq 0} b_q(n;j) t^j
    &=
    \frac{1}{\qfact{n}} \sum_{k=0}^{n-1} (-1)^kq^{\binom{k}{2}} \qbinom{n-1}{k}{q} \cdot \prod_{j = -k}^{n-1-k} \frac{1}{1-\qint{j}t}
    \\
    &=
    \sum_{k=0}^{n-1} \frac{(-1)^kq^{\binom{k}{2}}}{\qint{n}\qfact{k}\qfact{n-k-1}} \cdot \prod_{j = -k}^{n-1-k} \frac{1}{1-\qint{j}t}
    \label{eq:gen f for eval of hs}.
\end{align}
We must prove that the right-hand-sides of \cref{eq:thm-cat-conj,eq:gen f for eval of hs} are equal.
Clearing the denominator in the right-hand-side of \cref{eq:thm-cat-conj},
this is equivalent to proving the following equality between two polynomials of degree $n - 1$ in $t$:
\begin{equation} \label{eq:gen f for eval of hs: identity to prove}
    \sum_{k=0}^{n-1} \frac{(-1)^k q^{\binom{k}{2}}}{\qint{n}\qfact{k}\qfact{n-k-1}}
    \prod_{ \substack{ j \in [-(n-1),n-1] \\ j \notin [-k,n-1-k]}} (1 - \qint{j}t)
    = q^{-\binom{n}{2}} C_{n-1}(q) t^{n-1}.
\end{equation}
We do so by proving this equality when $t = \frac{1}{\qint{n-1-\ell}}$ for $\ell \in [0,n-1]$.
Observe that
\begin{align*}
    \prod_{j=k+1}^{n-1} \left( 1 - \frac{\qint{-j}}{\qint{n-1-\ell}} \right)
    &= \frac{1}{\qint{n-1-\ell}^{n-k-1}} \prod_{j=k+1}^{n-1} q^{-j}\qint{n-1-\ell+j}\\
    &= \frac{q^{-\binom{n}{2}+\binom{k+1}{2}}\qfact{2n-2-\ell}}{\qint{n-1-\ell}^{n-k-1} \qfact{n-1-\ell+k}}, \\
    \prod_{j=n-k}^{n-1} \left( 1 - \frac{\qint{j}}{\qint{n-1-\ell}} \right)
    &= \frac{1}{\qint{n-1-\ell}^{k}} \prod_{j=n-k}^{n-1} - q^{n-1-\ell}\qint{\ell+1-n+j} \\
    &= \frac{(-1)^k q^{k(n-1-\ell)}\qfact{\ell}}{\qint{n-1-\ell}^{k} \qfact{\ell-k}}.
\end{align*}
Thus, plugging $t = \frac{1}{\qint{n-1-\ell}}$ in the left-hand-side of \cref{eq:gen f for eval of hs: identity to prove} yields
\begin{align}
    \dfrac{q^{-\binom{n}{2}}}{\qint{n}\qint{n-1-\ell}^{n-1}} \sum_{k=0}^\ell q^{k(n-1-\ell+k) } \qbinom{2n-2-\ell}{n-1-k}{q} \qbinom{\ell}{k}{q}
    &= \dfrac{q^{-\binom{n}{2}}}{\qint{n}\qint{n-1-\ell}^{n-1}} \qbinom{2n-2}{n-1}{q} \\
    &= \dfrac{q^{-\binom{n}{2}}C_{n-1}(q)}{\qint{n-1-\ell}^{n-1}}, \label{eq:apply q chu vandermonde}
\end{align}
which is exactly the right-hand-side of \cref{eq:gen f for eval of hs: identity to prove} evaluated at $t = \frac{1}{\qint{n-1-\ell}}$.
The equality in \cref{eq:apply q chu vandermonde} follows by applying the $q$-Chu--Vandermonde identity \cref{eq:qVand1} with $a = \ell$, $b = 2n-2-\ell$, and $c = n-1$.
\end{proof}

Extracting the coefficient of the lowest-degree nonzero term in \cref{thm:cat conj}, we deduce the following result (see \cref{ex:def bq}).

\begin{corollary}
    For all $n \in \Z_{\geq 0}$,
    \begin{equation}
        \label{eq: identity h's}
        b_q(n;n-1) =
        q^{\binom{n - 2}{2} - 1} C_{n-1}(q^{-1}) = q^{-\binom{n}{2}} C_{n-1}(q).
    \end{equation}
\end{corollary}

More generally, we obtain the following.
\begin{corollary}
For all $n$ and $m$ in $\Z_{\geq 0}$,
\[
    b_q(n;n+m-1) =
    q^{-\binom{n}{2}} C_{n-1}(q) h_m\big(\qint{-(n-1)},\qint{-(n-2)},\ldots,\qint{n-1}\big).
\]
\end{corollary}

\subsection{Proof of \cref{thm:intro_q_summary_thm}(3): generating function for $c_q(n; p_1, \ldots, p_m)$}\label{ssec:q-c-proof}
We now prove \cref{thm:intro_q_summary_thm}(3), giving a $q$-analogue of the generating function for $c(n;p_1, \ldots, p_m)$.
Recall that
\[ c(n;p_1, \ldots,p_m) = [\sfc] \, \, e_{(p_1, \ldots, p_m)}(\Xi_n) \]
counts the number of ways to factor $\sfc$ into a product of $m$ permutations $\pi_1 \cdots \pi_m$, where $\pi_i$ has $n-p_i$ cycles. We are interested in the $q$-analogue
\begin{align}
    c_q(n; p_1,\ldots,p_m)
    :=& \, \, [T_{\sfc}] \,\, {\color{RoyalBlue} e_{(p_1,\ldots,p_m)}(\Xi_n(q))} \\
    =& \, \, [T_{\sfc}] \,\,
    {\color{RoyalBlue}
        \hskip5pt
        \smashoperator{
            \sum_{
                \substack{\mu^{(i)} \vdash n \\ \ell(\mu^{(i)})=n-p_i}
            }
        }
        \hskip3pt
        q^{-(p_1 + \cdots + p_m)} \, {\Gamma}_{\mu^{(1)}} \cdots {\Gamma}_{\mu^{(m)}}
    }.\label{eq: c in hecke}
\end{align}

Let ${\bf t} = (t_1,\ldots,t_m)$ and define
\begin{equation} \label{eq:F in terms of qE}
    F_{n}(q;{\bf t}) := [T_{\sfc}] \, \tilde{E}_n(q;t_1) \, \tilde{E}_n(q;t_2) \, \cdots \, \tilde{E}_n(q;t_m),
\end{equation}
where by \cref{def:etilde} and \cref{thm:ryba JM} we have
\[
    \tilde{E}_n(q;t)
    = \big(t + J_1(q)\big) \big(t + J_2(q)\big) \cdots \big(t + J_n(q)\big)
    = \sum_{i = 0}^{n} e_i\big(\Xi_n(q)\big) \, t^{n-i}.
\]
As in \cref{eq:blueprint-F-expanded},
the coefficient of $t_1^{n-p_1} \cdots t_m^{n-p_m}$
in $F_n(q, {\bf t})$ is
\begin{equation}
    [t_1^{n-p_1} \cdots t_m^{n-p_m}] \, F_n(q;{\bf t})
    = [T_{\sfc}] \, e_{(p_1, \ldots, p_m)}\big(\Xi_n(q)\big)
    = c_q(n; p_1, \ldots, p_m).
\end{equation}

Recall that our $q$-analogue of the binomial basis is
\begin{equation}\label{eq:fallfactorialdef}
\binom{t}{k}_q := \frac{1}{\qfact{k}} \prod_{i=0}^{k-1} (t-q^{-i}\qint{i}) = \frac{1}{\qfact{k}} \prod_{i=0}^{k-1} (t + \qint{-i}),
\end{equation}
and our $q$-analogue of $\M^{n}_{(r_1,\ldots,r_m)}$ is
\begin{equation}
\label{eq:def-Mqs}
\M^{n}_{(r_1,\ldots,r_m)}(q):= \# \bigg \{ (V_1,\ldots,V_m): V_i \subset \mathbb{F}_q^{n} \textrm{ with } \dim(V_i) = r_i \textrm{ and } V_1 + \cdots + V_m = \mathbb{F}_q^{n} \bigg\}.
\end{equation}

\begin{theorem} \label{thm:HeckeJacksonMs}
    For all $n$ and $m$ in $\Z_{\geq 0}$,
    the multivariate generating function $F_n(q; {\bf t})$ satisfies
    \begin{equation}
        F_n(q;{\bf t}) =
        \qfact{n}^{m-1} \ \ \smashoperator{\sum_{0 \leq r_1,\ldots,r_m \leq n-1}} \ \ q^{-\sum_i (n-r_i)r_i} \ \M^{n-1}_{(r_1,\ldots,r_m)}(q) \ \binom{t_1}{n-r_1}_q \ \binom{t_2}{n-r_2}_q \  \cdots \binom{t_m}{n-r_m}_q.
    \end{equation}

    In particular, when $\lambda$ is a composition of $n-1$,
    \begin{equation} \label{long-cycle-in-e-lambda-hecke-q}
        c_q(n; \lambda)
        = [T_{\sfc}]\,\, e_{\lambda}(\Xi_n(q))
        = q^{\sum_i \binom{\lambda_i}{2}-\binom{n}{2}} \frac{1}{\qint{n}} \prod_{i} \qbinom{n}{\lambda_i}{q}.
    \end{equation}
\end{theorem}

Before proving \cref{thm:HeckeJacksonMs}, we highlight
the $m=2$ case, which is particularly nice.

\begin{corollary} \label{cor:Jackson qMs 2 factors}
    For all $n \in \Z_{\geq 0}$,
    \begin{equation} \label{eq:JacksonMs 2 factors}
F_n(q;t_1,t_2) = \qfact{n} \sum_{0 \leq r_1, \ r_2 \leq n} q^{{a}} \qbinom{n-1}{r_1-1,r_2-1,n+1-r_1-r_2}{q} \binom{t_1}{r_1}_q \ \binom{t_2}{r_2}_q,
\end{equation}
where
\[ a=(r_1-1)(r_2-1)-(n-r_1)r_1-(n-r_2)r_2.\]
In particular, for $\lambda = (k,n-1-k)$ we have that
\begin{equation} \label{eq:qNarayana}
[T_{\sfc}]\,\, e_{k,n-1-k}(\Xi_n(q)) = q^{1-(k+1)(n-k)} N_{n,k+1}(q),
\end{equation}
where $N_{n,k+1}(q)$ is the $q$-Narayana number
\[N_{n,k+1}(q):=\dfrac{1}{\qint{n}}\qbinom{n}{k+1}{q}\qbinom{n}{k}{q}.\]
\end{corollary}

\begin{proof}
Applying \cref{thm:HeckeJacksonMs} for $m=2$ and reindexing the right-hand size by $r_i$ to $n-r_i$, we obtain
\begin{align*}
    F_n(q;t_1,t_2) &= \qfact{n} \sum_{0 \leq r_1, \ r_2 \leq n} q^{(n-r_1)r_1+(n-r_2)r_2} \ \M^{n-1}_{(n-r_1,n-r_2)}(q) \cdot \binom{t_1}{r_1}_q \binom{t_2}{r_2}_q.
\end{align*}
The claimed generating polynomial in \cref{eq:JacksonMs 2 factors}  then follows by \cref{cor:qMs m=2}. When $\lambda=(n-1-k,k)$, we apply the formula in \cref{long-cycle-in-e-lambda-hecke-q} and obtain \cref{eq:qNarayana}, as desired.
\end{proof}

The proof of \cref{thm:introsymmetricsummary}(3) outlined in
\S\ref{sec:prototypical-example-isotypic-projectors-technique} serves as a blueprint for our proof of \cref{thm:HeckeJacksonMs}. That proof relied on the Chu--Vandermonde identity~\cref{eq:chu-vandermonde}, which
describes the change of basis of the polynomial ring $\Z[t]$ between the basis
of rising factorials
\[ (t - k)^{(n)} = (t-k)(t-k+1) \cdots (t) (t+1) \cdots (t+n-k-1) = \prod_{j= -k}^{n-k-1} \big(t + j\big), \] and the binomial basis
\[ \binom{t}{k} = \frac{1}{k!} \prod_{i=0}^{k-1} (t-i).  \]

Here we develop a $q$-analogue of this change of basis (\cref{prop:etilde}) that we will use in the
proof of \cref{thm:HeckeJacksonMs}. As discussed above, our $q$-analogue of the binomial basis is given in \cref{eq:fallfactorialdef}. Our analogue of the rising factorial basis $(t-k)^{(n)}$ will be

\begin{equation}\label{eq:etildedef}
    \tilde{E}_{n,k}(q;t)
    := \prod_{c \in \qcont((n-k,1^k))} \big(t + c\big)
    = \prod_{j= -k}^{n-k-1} \big(t + \qint{j}\big),
\end{equation}
for $n > 0$ and $0 \leq k < n$. Note that $\tilde{E}_{n,k}(1;t) = (t-k)^{(n)}$.

We now give an explicit expression for the $\tilde{E}_{n,k}(q;t)$ in
terms of the $\binom{t}{j}_q$.

\begin{proposition}\label{prop:etilde}
    For all $n > 0$ and $0 \leq k < n$,
    \begin{equation}
    \label{eq-prop:etilde}
        \tilde{E}_{n,k}(q;t)
        = \qfact{n} \sum_{j = k+1}^n q^{-(n-j)j} \qbinom{n-1-k}{n-j}{q} \binom{t}{j}_q.
    \end{equation}
\end{proposition}
\begin{proof}
Note that both the right and left-hand-sides are polynomials of degree $n$ in $t$,
so it is sufficient to check they agree on $n+1$ distinct values of $t$.

To do so, we verify the identity for $t = -\qint{-m}$ and any $m > n$. We claim that for $m > n$, both sides evaluated at $t = -\qint{-m}$ are equal to
\begin{equation}\label{eq:rewriteEtilde}
q^{-m n} \qfact{n-k-1} \qbinom{m+n-k-1}{m}{q} \qint{m} \qint{m-1} \cdots \qint{m-k}.
\end{equation}
To prove this, we begin with the left-hand-side of \cref{eq-prop:etilde}.
Note that
\[ - \qint{-m} + \qint{j} = q^{-m} \qint{m+j}, \] so that
\begin{align*}
    \tilde{E}_{n,k}(q;-\qint{-m})
    &= \prod_{j=-k}^{n-k-1} q^{-m} \qint{m+j} = q^{-m n} \left( \prod_{j=1}^{n-k-1} \qint{m+j} \right) \left(  \prod_{j=-k}^{0} \qint{m+j} \right) \\
    &= q^{-m n} \qfact{n-k-1}  \qbinom{m+n-k-1}{m}{q} \qint{m} \qint{m-1} \cdots \qint{m-k}.
\end{align*}
On the other hand, evaluating $\binom{t}{j}_q$ at $t = -\qint{-m}$ gives
\[
    \binom{-\qint{-m}}{j}_q
    = \frac{1}{\qfact{j}} \prod_{i=0}^{j-1} ( \qint{-i} - \qint{-m} )
    = \frac{1}{\qfact{j}} \prod_{i=0}^{j-1} q^{-m} \qint{m-i}
    = \frac{q^{-m j}}{\qfact{j}} ,\qint{m} \qint{m-1} \cdots \qint{m-(j-1)},
\]
and therefore evaluating the right-hand-side of \cref{eq-prop:etilde} at $t = -\qint{-m}$, we obtain
\begin{multline*}
    \qfact{n} \sum_{j = k+1}^n q^{-(n-j)j} \qbinom{n-1-k}{n-j}{j} \binom{-\qint{-m}}{j}_q \\
    \begin{aligned}[t]
        &= \qfact{n}  \sum_{j = k+1}^n q^{-(n-j)j} \qbinom{n-1-k}{n-j}{j} \left( \frac{1}{\qfact{j}}\prod_{i=0}^{j-1} q^{-m} \qint{m-i}\right) \\
        &= \sum_{j = k+1}^n q^{-(n-j)j -m j}
            \frac{\qfact{n} \qfact{n-1-k} \qint{m-(j-1)} \cdots \qint{m}}{\qfact{n-j} \qfact{j-k-1} \qfact{j}}   \\
        &= \qfact{n-1-k} \left( \sum_{j = k+1}^n q^{-(n-j)j -m j} \qbinom{n}{j}{q} \qbinom{m-(k+1)}{m-j}{q} \right) \qint{m-k} \qint{m-k-1} \cdots \qint{m} \\
        &= q^{-m n} \qfact{n-k-1}  \qbinom{m+n-k-1}{m}{q} \qint{m-k} \cdots \qint{m}.
    \end{aligned}
\end{multline*}
The last equality is the $q$-Chu--Vandermonde formula \eqref{eq:qVand2} with $a = m-k-1$, $b = n$, and $c = m$.
\end{proof}

We are now ready to prove \cref{thm:HeckeJacksonMs}.
\begin{proof}[Proof of \cref{thm:HeckeJacksonMs}] We will make use of two results
about the coefficients $\M^{(n-1)}_{(r_1, r_2, \ldots, r_m)}(q)$
whose proofs we postpone to \S\ref{sssec:interpretation-and-combinatorics-of-Ms},
where we study their combinatorial properties
(\cref{prop:explicit-formula-Mq} and~\cref{lem:qMs-qmultinomial}).

Before we do so, we imitate \cref{eq:blueprint-F-expanded}
and \cref{eq:blueprint-F-using-tilde-E},
applying \cref{prop:isoprojtechnique} to \cref{eq:F in terms of qE} to conclude
\begin{equation} \label{eq: qF in terms of etildes}
    F_n(q;{\bf t}) = \frac{1}{\qfact{n}} \sum_{k=0}^{n-1} (-1)^k q^{\binom{k}{2}} \qbinom{n-1}{k}{q}
    \tilde{E}_{n,k}(q;t_1) \cdots \tilde{E}_{n,k}(q;t_m).
\end{equation}

Next, we use \cref{prop:etilde} to rewrite $\tilde{E}_{n,k}(q;t_i)$ in the basis $\left\{\binom{t_i}{n-r_i}_q \right\}$:
\begin{multline*}
    F_n(q;{\bf t}) =
    \qfact{n}^{m-1} \sum_{k=0}^{n-1} (-1)^k q^{\binom{k}{2}} \qbinom{n-1}{k}{q}
        {\sum_{0 \leq r_1,\ldots,r_m \leq n-k-1}} \quad  \prod_i \left( q^{-(n-r_i)r_i} \qbinom{n-1-k}{r_i}{q} \binom{t}{n-r_i}_q\right) \\ =
    \qfact{n}^{m-1} \quad \smashoperator{\sum_{0 \leq r_1,\ldots,r_m \leq n-1}} \;\; q^{-\sum_i (n-r_i)r_i}
         \underbrace{\bigg( \smashoperator[r]{\sum_{k=0}^{n-1-\max_i\{r_i\}}} \;\; (-1)^k q^{\binom{k}{2}} \qbinom{n-1}{k}{q} \prod_i \qbinom{n-1-k}{r_i}{q} \bigg)}_{:=(*)}
            \binom{t}{n-r_1}_q \!\! \cdots \binom{t}{n-r_m}_q.
\end{multline*}
By \cref{prop:explicit-formula-Mq}, the summations $(*)$ in parentheses above are equal to
\[
    (*) = \sum_{k=0}^{n-1-\max_i\{r_i\}} (-1)^k q^{\binom{k}{2}} \qbinom{n-1}{k}{q} \prod_i \qbinom{n-1-k}{r_i}{q}
    = \M^{n-1}_{(r_1,\ldots,r_m)}(q).
\]
Thus,
\[
    F_n(q;{\bf t}) =
        \qfact{n}^{m-1} \sum_{0 \leq r_1,\ldots,r_m \leq n-1} q^{-\sum_i (n-r_i)r_i} \ \M^{n-1}_{(r_1,\ldots,r_m)}(q) \
            \binom{t_1}{n-r_1}_q \binom{t_2}{n-r_2}_q \cdots \binom{t_m}{n-r_m}_q,
\]
which proves the first statement of the theorem.

To prove the second statement,
let $\lambda = (\lambda_1,\ldots,\lambda_m)$ be a partition of $n-1$.
Since $\M^{n-1}_{(r_1,\ldots,r_m)} = 0$ whenever $r_1 + \cdots + r_m < n-1$ and the degree of $\binom{t}{j}_q$ in $t$ is $j$, we have
\[
    [{\bf t}^{{\bf n}-\lambda}] F_n(q;{\bf t})
    = \qfact{n}^{m-1} q^{-\sum_i (n-\lambda_i)\lambda_i}  \ \M^{n-1}_{(\lambda_1,\ldots,\lambda_m)}(q) \
            [{\bf t}^{{\bf n}-\lambda}]\binom{t_1}{n-\lambda_1}_q \binom{t_2}{n-\lambda_2}_q \cdots \binom{t_m}{n-\lambda_m}_q.
\]
Now, using \cref{lem:qMs-qmultinomial} and the fact that the leading coefficient of $\binom{t_i}{n-\lambda_i}_{q}$ is
$ \frac{1}{\qfact{n-\lambda_i}}$,
we have
\begin{align*}
    [T_{\sfc}]\,\, e_{\lambda}(\Xi_n)
    &= [{\bf t}^{{\bf n}-\lambda}] F_n(q;{\bf t}) \\
    &= \qfact{n}^{m-1} \; q^{-\sum_i (n-\lambda_i)\lambda_i + \binom{n-1}{2} - \sum_i \binom{\lambda_i}{2}} \qbinom{n-1}{\lambda_1,\ldots,\lambda_m}{q} \frac{1}{\qfact{n-\lambda_1}\cdots \qfact{n-\lambda_m}} \\
    &= q^{\sum_i \binom{\lambda_i}{2} - \binom{n}{2}} \frac{1}{\qint{n}} \prod_{i} \qbinom{n}{\lambda_i}{q},
\end{align*}
completing the proof of \cref{thm:HeckeJacksonMs}.
\end{proof}

\subsection{Evaluations of other symmetric functions}
\label{ssec:evaluations-other-symmetric-functions}

As highlighted in the introduction, an important consequence of
\cref{thm:HeckeJacksonMs} is the explicit link it provides between the
coefficient of $T_{\sfc}$ in $f(\Xi_n(q))$ and the $q$-principal
specialization of $f$ for \emph{any} homogeneous symmetric function $f$ of
degree $n - 1$.
We prove this (i.e. \cref{thm:intro_reciprocity} from the Introduction) here and use it to determine the coefficient of
$T_{\sfc}$ in $h_{\lambda}(\Xi_n(q))$,
$p_{\lambda}(\Xi_n(q))$, and $s_{\lambda}(\Xi_n(q))$.

\begin{remark}\rm
We learned of the connection between principal specializations of $f$ and the evaluations $f(\Xi_n)$ for the symmetric group via private communication from John Irving, and are grateful to him for pointing out this insight.
Notably, his argument can be extended to the Hecke algebra using our
\cref{thm:HeckeJacksonMs}, as shown below.

\end{remark}
Given a symmetric function $f \in \Lambda$, the
\emph{$q$-principal specialization} of $f$, denoted $ps_f(n; q)$, is
defined as
\begin{equation*}
    ps_f(n; q) := f(1, q, q^2, \ldots, q^{n-1}).
\end{equation*}
The first step is to relate \cref{thm:HeckeJacksonMs} to $q$-principal specializations.

\begin{theorem}
    \label{thm: key lemma q reciprocity}
    Let $f$ be a homogeneous symmetric function of degree $n-1$. Then
    \begin{eqnarray}
        [T_{\sfc}] \, f(\Xi_n(q))
        & = &
        \frac{q^{-\binom{n}{2}}}{\qint{n}} \ ps_f(n; q).
        \label{eq: cycle coeff Hecke JM eval f as psq of f}
    \end{eqnarray}
\end{theorem}

\begin{proof}
    Since both sides of \cref{eq: cycle coeff Hecke JM eval f as psq of f}
    are linear in $f$, it suffices to prove both sides coincide
    on a basis of the space $\Lambda_{n-1}$ of homogeneous symmetric functions of degree $n-1$.
    We will do so for the basis of $\Lambda_{n-1}$ given by the elementary symmetric functions $\{e_\lambda\}_{\lambda \vdash n{-}1}$.

    Let $\lambda$ be a partition of $n-1$.     By \cite[Prop. 7.8.3]{EC2}, the $q$-principal specialization of $e_{\lambda}$ is
    \begin{equation*}
        e_{\lambda}(1,q,\ldots,q^{n-1})=\prod_i e_{\lambda_i}(1,q,\ldots,q^{n-1}) = \prod_i q^{\binom{\lambda_i}{2}}\qbinom{n}{\lambda_i}{q},
    \end{equation*}
    On the other hand, by \cref{long-cycle-in-e-lambda-hecke-q} in
    \cref{thm:HeckeJacksonMs}, we have that
    \begin{equation}
        [T_{\sfc}]\,\, e_{\lambda}(\Xi_n(q))
        =
        q^{\sum_i \binom{\lambda_i}{2}-\binom{n}{2}} \frac{1}{\qint{n}} \prod_{i} \qbinom{n}{\lambda_i}{q}.
    \end{equation}
We can thus substitute in $ps_{e_{\lambda}}(n;q)$ to obtain
    \begin{equation}
        [T_{\sfc}]\,\, e_{\lambda}(\Xi_n(q))
        = \frac{q^{-\binom{n}{2}}}{\qint{n}} \ ps_{e_{\lambda}}(n;q),
    \end{equation}
    as desired.
\end{proof}

Next, we use \cref{thm: key lemma q reciprocity} to determine the
coefficient of $T_{\sfc}$ in the evaluations $h_{\lambda}(\Xi_n(q))$,
$p_{\lambda}(\Xi_n(q))$, and $s_{\lambda}(\Xi_n(q))$.

\begin{theorem}
    \label{thm:other-symmetric-functions}
    Let $\lambda \vdash n - 1$. Then
    \begin{eqnarray}
        [T_{\sfc}] \,\, h_{\lambda}(\Xi_n(q))
        & = &
        \frac{q^{-\binom{n}{2}}}{\qint{n}} \, \prod_{i = 1}^{\ell(\lambda)} \qbinom{n+\lambda_i-1}{\lambda_i}{q},
        \label{long-cycle-in-h-lambda-hecke-q}
        \\\null
        [T_{\sfc}] \,\, p_{\lambda}(\Xi_n(q))
        & = &
        \frac{q^{-\binom{n}{2}}}{\qint{n}} \, \prod_{i=1}^{\ell(\lambda)} \frac{1-q^{n\lambda_i}}{1-q^{\lambda_i}},
        \label{long-cycle-in-p-lambda-hecke-q}
        \\\null
        [T_{\sfc}] \,\, s_{\lambda}(\Xi_n(q))
        & = &
        \frac{q^{b(\lambda)-\binom{n}{2}}}{\qint{n}} \, \prod_{(i,j)\in \lambda} \frac{1-q^{n+j-i}}{1-q^{h(i,j)}},
        \label{long-cycle-in-s-lambda-hecke-q}
    \end{eqnarray}
    where $b(\lambda)=\sum_{i=1}^{\ell(\lambda)} (i-1)\lambda_i$,
    and $h(i,j) = \lambda_i+\lambda'_j-i-j+1$ is the hook-length of the cell $(i,j)$ in $\lambda$.
\end{theorem}

\begin{proof}
    \cref{long-cycle-in-h-lambda-hecke-q}
    follows from \cref{thm: key lemma q reciprocity}
    with $f = h_{\lambda}$ since
    \begin{equation*}
        ps_{h_k}(n; q)
        = h_k(1,q,\ldots,q^{n-1})
        = \qbinom{n+k-1}{k}{q}.
    \end{equation*}

    \cref{long-cycle-in-p-lambda-hecke-q}
    follows from \cref{thm: key lemma q reciprocity}
    with $f = p_{\lambda}$ since
    \begin{equation*}
        ps_{p_k}(n;q)
        = p_k(1,q,\ldots,q^{n-1})
        = \frac{1-q^{nk}}{1-q^k}
    \end{equation*}
    by~\cite[Prop.~7.8.3]{EC2}.

    \cref{long-cycle-in-s-lambda-hecke-q}
    follows from \cref{thm: key lemma q reciprocity}
    with $f = s_{\lambda}$ since
    \begin{equation*}
        ps_{s_{\lambda}}(n;q)
        = s_{\lambda}(1,q,\ldots,q^{n-1})
        = q^{b(\lambda)} \prod_{(i,j) \in \lambda} \frac{1-q^{n+j-i}}{1-q^{h(i,j)}}
    \end{equation*}
    by~\cite[Thm.~7.21.2]{EC2},
    where $h(i,j)=\lambda_i+\lambda'_j-i-j+1$ is the hook-length of the cell $(i,j)$ in $\lambda$.
\end{proof}

\section{Interpretation and combinatorics of $\M^n_{\bf r}(q)$}
\label{sssec:interpretation-and-combinatorics-of-Ms}
In this section, we study the numbers $\M^n_{\bf r}(q)$ that appear in \cref{eq:def-Mqs}.
Given $q$ a prime power and integers $n \in \Z_{\geq 0}$ and ${\bf r} = (r_1, \ldots, r_m) \in \Z_{\geq 0}^m$,
let
\[ \MM_{\bf r}^n := \Big \{ (V_1,\ldots,V_m): V_i \subset \mathbb{F}_q^{n} \textrm{ with } \dim(V_i) = r_i \textrm{ and } V_1 + \cdots + V_m = \mathbb{F}_q^{n} \Big\},\]
so that, by definition, $\M^{n}_{\bf r}(q):= \# \MM_{\bf r}^n$.

While the $\M^n_{\bf r}(q)$ appear in \cref{thm:HeckeJacksonMs}, it is not evident that they
\begin{enumerate}
    \item lie in $\Z_{\geq 0}[q]$ (as a function of $q$), or
    \item are a $q$-analogue of the numbers $\M^{n}_{\bf r}$ in \cref{thm:introsymmetricsummary}(1), in the sense that they \emph{$q$-count} the objects in $\M_{\bf r}^n$ with an appropriate statistic.
\end{enumerate}
We will show that both of these things are true (see \cref{prop:positive recurrence Ms,prop:M-via-q-statistic}, respectively), as well as several recurrences and formulas for $\M^n_{\bf r}(q)$.

We begin by showing that $\M^n_{\bf r}(q)$ has an elegant alternating sum formula.
To do so, we consider the bilinear form $\langle x , y \rangle = x^t y$ on $\F_q^n$,
with respect to which we take orthogonal subspaces $W^{\perp}$ for any $W \subset \F_q^n$.
The map $W \to W^\perp$ is inclusion reversing and satisfies $\dim(W) + \dim(W^\perp) = n$ (even if $W + W^\perp \neq \F_q^n$).

\begin{proposition}
\label{prop:explicit-formula-Mq}
For positive integers $n \in \Z_{\geq 0}$ and ${\bf r} \in \Z_{\geq 0}^m$, and any prime power $q$, we have
\begin{equation}
    \label{eq:formula-Mq}
    \M^n_{\bf r}(q) = \sum_{k=0}^{n-\max(r_i)} (-1)^k q^{\binom{k}{2}} \qbinom{n}{k}{q}\prod_{i=1}^m \qbinom{n-k}{r_i}{q}.
\end{equation}
Thus, $\M^n_{\bf r}(q)$ is a polynomial function of $q$.
\end{proposition}

\begin{proof}
Fix $n \in \Z_{\geq 0}$ and ${\bf r} \in \Z_{\geq 0}^m$.
Define functions $f,g$ on the lattice of subspaces of $\F_q^n$ as follows.
For a subspace $W \subseteq \F_q^n$,
\begin{align*}
        f(W) & = \#\big\{ (V_1,\ldots,V_m) \mid \dim V_i = r_i \text{ and } (V_1 + \cdots + V_m)^\perp = W \big\}, \text{ and } \\
        g(W) & = \#\big\{ (V_1,\ldots,V_m) \mid \dim V_i = r_i \text{ and } (V_1 + \cdots + V_m)^\perp \supseteq W \big\}.
\end{align*}
By definition, $\M^n_{\bf r}(q) = f(\{ 0 \})$.
Moreover, note that for all subspaces $W \subseteq \mathbb{F}_q^n$,
\[ g(W) = \sum_{U \supseteq W} f (U). \]
Therefore, by Möbius inversion,
\[ f (W) = \sum_{U \supseteq W} \mu(W,U) g(U) \] for all $W \subseteq \mathbb{F}_q^n$.
The Möbius function of the lattice of subspaces of $\mathbb{F}_q^n$ has the following explicit formula:
\[ \mu(W,U) = (-1)^{\dim U/W} q^{\binom{\dim U/W}{2}}. \]
Finally, note that $(V_1 + \cdots + V_m)^\perp \supseteq W$ if and only if $V_i \subseteq W^\perp$ for all $i$.
Since $\dim W^\perp = n - \dim W$, we deduce
\[
    g(W) = \prod_{i=1}^m \qbinom{n-\dim W}{r_i}{q}.
\]
We thus conclude that
\[
    \M^n_{\bf r}(q) = f (\{0\})
    = \sum_{U \subseteq \mathbb{F}_q^n} (-1)^{\dim U} q^{\binom{\dim U}{2}} g(U)
    = \sum_{k=0}^n (-1)^{k} q^{\binom{k}{2}} \qbinom{n}{k}{q} \prod_{i=1}^m \qbinom{n - k}{r_i}{q}.
\]
In the last step we grouped subspaces $U \subseteq \mathbb{F}_q^n$ by their dimension. Lastly, since $q$-binomial coefficients are polynomials in $q$, it follows that $\M^n_{\bf r}(q)$ is also a polynomial function in $q$ for fixed $n \in \Z_{\geq 0}$ and ${\bf r} \in \Z^m_{\geq 0}$.
\end{proof}

\begin{remark}
    Note that in principle, $\M^n_{\bf r}(q)$ is defined only when $q$ is a prime power. However, based on the polynomiality result above (\cref{prop:explicit-formula-Mq}), by abuse of notation, we will also use $\M^n_{\bf r}(q)$ to denote this polynomial, treating $q$ as an indeterminant. 
\end{remark}

Comparing \cref{eq:formula-Mq,eq:rel Ms}, we deduce that the $\M^n_{\bf r}(q)$ are indeed a $q$-analogue of the numbers $\M^n_{\bf r}$.

\begin{corollary}
\label{cor:Mq-are-q-analogues}
For all $n \in \mathbb{Z}_{\geq 0}$ and ${\bf r} \in \mathbb{Z}_{\geq 0}^m$,
\[
    \M^n_{\bf r}(1) = \M^n_{\bf r}.
\]
\end{corollary}

The following surprising cancellation is a direct result of the definition of $\M_{\bf r}^n(q)$.

\begin{corollary}
If $r_1 + \cdots + r_m < n$, then
\[ \sum_{k=0}^{n-\max(r_i)} (-1)^k q^{\binom{k}{2}} \qbinom{n}{k}{q}\prod_{i=1}^m \qbinom{n-k}{r_i}{q} = 0. \]
\end{corollary}

\begin{proof}
If $r_1 + \cdots + r_m < n$, then a tuple of subspaces $(V_1,\ldots,V_m)$ of $\mathbb{F}_q^n$ with $\dim(V_i)=r_i$ cannot sum to $\mathbb{F}_q^n$. Therefore, $\M_{\bf r}^n(q) = 0$ and the result follows by \cref{prop:explicit-formula-Mq}.
\end{proof}

Next, in \cref{lem:qMs-qmultinomial,cor:qMs m=2}, we give explicit closed formulas for two special cases of $\M_{\bf r}^n(q)$.

\begin{proposition}  \label{lem:qMs-qmultinomial}
If ${\bf r}= (r_1, \ldots, r_m) \in \Z_{\geq 0}^m$ is a composition of $n$ (that is, if $r_1 + \cdots + r_m = n$), then
\[  \M^n_{\bf r}(q) = q^{\binom{n}{2}-\sum_i \binom{r_i}{2}} \qbinom{n}{r_1,\ldots,r_m}{q}. \]
\end{proposition}

\begin{proof}
Suppose that $r_1 + \cdots + r_m = n$.
In this case $V_1 + V_2 + \cdots + V_m = \F_q^n$ is necessarily a direct sum.
Therefore,
\[
    \M^n_{\bf r}(q) = \# \Big( \GL_n(\mathbb{F}_q) / \GL_{\bf r}(\F_q) \Big),
\]
where
\[ \GL_{\bf r}(\F_q) := \GL_{r_1}(\mathbb{F}_q) \times \cdots \times \GL_{r_m}(\mathbb{F}_q) \subset \GL_n(\mathbb{F}_q).\]

On the other hand, by the usual interpretation of the $q$-multinomial coefficients,
\[ \qbinom{n}{\bf r}{q} = \qbinom{n}{r_1,\ldots,r_m}{q} =  \# \big\{ U_1 \subset U_2 \subset \cdots \subset U_m = \F_q^n: \dim U_k = r_1 + \cdots + r_k  \big\}.\]
    Thus,
    \[
        \qbinom{n}{\bf r}{q} = \# \Big( \GL_n(\mathbb{F}_q)  / G({\bf r}) \Big),
    \]
    where $G({\bf r}) \subset \GL_n(\mathbb{F}_q)$ is the subgroup of block upper triangular matrices of the form
    \begin{equation}\label{flagmatrix}
        \begin{pmatrix}
            \GL_{r_1}(\mathbb{F}_q) & * &    & * \\
            0 & \GL_{r_2}(\mathbb{F}_q)  &   & * \\
            & & \ddots & \\
            0 & 0 &  & \GL_{r_m}(\mathbb{F}_q)
        \end{pmatrix}.
    \end{equation}
    Let $a$ be the number of free entries in \cref{flagmatrix}, so that
    \begin{align*}
        a &=  \dim_{\F_q}\Big( G({\bf r})/ \GL_{\bf r}(\F_q) \Big) \\
        &= \binom{n}{2} - \sum_{i=1}^m \binom{r_i}{2}.
    \end{align*}
    We can therefore conclude that
    \begin{align*}
        \M_{\bf r}^n(q) &= \# \Big( \GL_n(\mathbb{F}_q)  / \GL_{\bf r}(\F_q) \Big) \\
        &= \# \Big( \GL_n(\mathbb{F}_q)  / G({\bf r}) \Big)  \cdot \# \Big( G({\bf r}) / \GL_{\bf r}(\F_q) \Big)  \\
        &= \qbinom{n}{\bf r}{q} \ q^a. \qedhere
    \end{align*}
\end{proof}

The case $m=2$ has a closed formula for any possible ${\bf r} \in \Z_{\geq 0}^2$

\begin{proposition}
\label{cor:qMs m=2}
For all $n,r_1,r_2 \in \mathbb{Z}_{\geq 0}$, we have that
\[
\M^n_{(n-r_1,n-r_2)}(q) = q^{r_1r_2} \qbinom{n}{r_1,r_2,n-r_1-r_2}{q}.
\]
\end{proposition}

\begin{proof}
First observe that both sides of the equation are zero unless $0 \leq r_1,r_2 \leq n$ and $r_1 + r_2 \leq n$;
so we assume this is the case.
By definition $\M^n_{(n-r_1,n-r_2)}(q)$ is the cardinality of 
\[ \MM^n_{(n-r_1,n-r_2)} = \big\{ (V_1, V_2): V_i \subseteq \F_q^n, \ \  \dim(V_i) = n-r_i,  \ \ \textrm{ and } V_1 + V_2 = \F_q^n. \big\},\]
while the multinomial coefficient $\qbinom{n}{r_1,r_2,n-r_1-r_2}{q}$ is the cardinality of
\[ \mathcal{N}^n_{(n-r_1,n-r_2)} := \big \{ U \subseteq W \subseteq \mathbb{F}_q^n:  \dim(U) = n-r_1-r_2 \ \ \textrm{ and }  \ \ \dim(W) = \dim(U) + r_2 = n-r_1\big\}.\]

Consider the map $\varphi: \MM_{(n-r_1,n-r_2)}^n \to \mathcal{N}_{(n-r_1,n-r_2)}^n$ defined by
\[
    \varphi\big((V_1,V_2)\big) = (V_1 \cap V_2 \subseteq V_1).
\]
This map is well-defined since
\[
    \dim(V_1 \cap V_2) = \dim(V_1) + \dim(V_2) - \dim(V_1 + V_2) = n - r_1 - r_2,
\]
and is clearly surjective.
We claim the fibers of $\varphi$ have cardinality $q^{r_1r_2}$, from which the theorem follows:
\[ \M^n_{(n-r_1,n-r_2)}(q) = \# \MM_{(n-r_1,n-r_2)}^n = q^{r_1 r_2} \#\mathcal{N}_{(n-r_1,n-r_2)}^n = q^{r_1 r_2} \qbinom{n}{r_1,r_2,n-r_1-r_2}{q}.\]

Note that by symmetry (i.e. via a $\GL_n(\mathbb{F}_q)$-action), we can assume $(U\subseteq W) \in \mathcal{N}^n_{(n-r_1,n-r_2)}$
is such that $U$ is spanned by the first $n-r_1$ canonical basis vectors,
and $W$ is spanned by $U$ and the last $r_2$ canonical basis vectors.
Let $(V_1,V_2) \in \varphi^{-1}\big((U\subseteq W)\big)$.
Thus, modulo $U = V_1 \cap V_2$ and row operations, $V_2$ is spanned by the rows of a unique matrix of the form
\[
    \begin{pmatrix}
        0 & \cdots & 0 & 1 & 0 & \cdots & 0 & * & \cdots & * \\
        0 & \cdots & 0 & 0 & 1 & \cdots & 0 & * & \cdots & * \\
        \vdots & \ddots & \vdots & \vdots & \vdots & \ddots & \vdots & \vdots & \ddots & \vdots \\
        0 & \cdots & 0 & 0 & 0 & \cdots & 1 & * & \cdots & *.
    \end{pmatrix}
\]
There are exactly $r_1 r_2$ matrices of this form (i.e. free entries), from which the claim follows.
\end{proof}

Next, we give a recurrence for $\M^{n}_{\bf r}(q)$ that specializes to \cref{eq:key recurrence M} when $q = 1$. Interestingly, \cref{eq:recursion-qMs} below has an additional term that vanishes when $q=1$.

\begin{proposition}\label{prop:recursion-qMs}
    The $\M^n_{\bf r}(q)$ satisfy the following recurrence:
    \begin{equation}
        \label{eq:recursion-qMs}
        \M^n_{\bf r}(q) = \left( q^{\sum_i r_i} - q^{n-1} \right) \M^{n-1}_{\bf r}(q) + \sum_{\varnothing \neq T \subseteq [m]} \left( q^{\sum_{i \notin T} r_i} \right) \M^{n-1}_{{\bf r}-{\bf e}_T}(q),
    \end{equation}
    where ${\bf e}_T \in \mathbb{Z}^m$ denotes the indicator vector of $T \subseteq [m]$.
\end{proposition}

\begin{proof}
Recall once again that $\M_{\bf r}^n(q)$ is defined to be the cardinality of the set
\[ \MM^n_{\bf r} :=  \Big\{ (V_1,\ldots,V_m): V_i \subset \mathbb{F}_q^{n} \textrm{ with } \dim(V_i) = r_i \textrm{ and } V_1 + \cdots + V_m = \mathbb{F}_q^{n} \Big\}.\]

Fix a one-dimensional subspace $L \subseteq \mathbb{F}_q^n$
and let $\pi_L$ denote the projection $\pi_L : \mathbb{F}_q^n \to \mathbb{F}_q^n/L \cong \mathbb{F}_q^{n-1}$.
Observe that if $(V_1,\ldots,V_m) \in \MM_{\bf r}^n$, then
$\pi_L(V_1) + \cdots + \pi_L(V_m) = \pi_L(\mathbb{F}_q^n) = \mathbb{F}_q^n/L \cong \mathbb{F}_q^{n-1}$ and
\begin{equation*}
    \dim \pi_L(V_i) =
    \begin{cases}
        r_i - 1 & \text{if } L \subseteq V_i,\\
        r_i & \text{if } L \not\subseteq V_i.
    \end{cases}
\end{equation*}
That is, $\big(\pi_L(V_1),\ldots,\pi_L(V_m)\big) \in \MM^{n-1}_{{\bf r}- {\bf e}_T}$,
where $T = \big \{ i \in [m] \mid L \subseteq V_i \big \}$.

Let $\eta_L$ be the following map:
\begin{align*}
    \eta_L: \MM_{\bf r}^n &\longrightarrow \bigcup_{T \subseteq [m]} \MM^{n-1}_{{\bf r}- {\bf e}_T} \\
     (V_1,\ldots,V_m) &\longmapsto \big(\pi_L(V_1),\ldots,\pi_L(V_m)\big).
\end{align*}

We prove \cref{eq:recursion-qMs} by computing the cardinality of the fibers of $\eta_L$.

Given $(U_1,\ldots,U_m) \in \MM^{n-1}_{{\bf r}- {\bf e}_T}$ for some $T \subseteq [m]$,
we say an $r_i$-dimensional subspace $V_i \subseteq \mathbb{F}_q^n$ \emph{lifts} $U_i$ if $\pi_L(V_i) = U_i$.
We also say that a tuple $(V_1,\ldots,V_m)$ \emph{lifts} $(U_1,\ldots,U_m)$ if $V_i$ lifts $U_i$ for each $i$.
Thus, the fiber $\eta_L^{-1}\big( (U_1,\ldots,U_m) \big)$ consists precisely of those lifts that are in $\MM_{\bf r}^n$;
i.e. that span all $\mathbb{F}_q^n$.

We claim that \cref{eq:recursion-qMs} follows immediately from the following:
\begin{enumerate}
\item Any $(U_1,\ldots,U_m) \in \MM^{n-1}_{{\bf r}- {\bf e}_T}$ has exactly $q^{\sum_{i \notin T} r_i}$ lifts.
\item If $T \neq \varnothing$, all lifts of $(U_1,\ldots,U_m) \in \MM^{n-1}_{{\bf r}- {\bf e}_T}$ are in $\MM_{\bf r}^n$.
\item (If $T = \varnothing$,) exactly $q^{n-1}$ lifts of $(U_1,\ldots,U_m) \in \MM^{n-1}_{{\bf r}}$ are not in $\MM_{\bf r}^n$.
\end{enumerate}

To prove the first claim, note that if $\dim(U_i) = r_i - 1$ (equivalently, if $i \in T$),
then $U_i$ has a unique lift $V_i = \pi_L^{-1}(U_i)$.
On the other hand, if $\dim(U_i) = r_i$ (equivalently, if $i \not \in T$), then any $r_i$-dimensional subspace $V_i \subseteq \pi_L^{-1}(U_i)$ not containing $L$ is lifts $U_i$.
In the latter case, where $\dim(U_i) = r_i$, we have
\[
    \#\big\{ V_i \subseteq \pi_L^{-1} (U_i) : \dim V_i = r_i, \  L \not\subseteq V_i \big\} = [r_i + 1]_q - \qint{r_i} = q^{r_i}.
\]
The first claim then follows since the components $V_i$ of an arbitrary lift of $(U_1,\ldots,U_m) \in \MM^{n-1}_{{\bf r}- {\bf e}_T}$ can be chosen independently.

Now, assume $T \neq \varnothing$ and let $(V_1,\ldots,V_m)$ be a lift of $(U_1,\ldots,U_m) \in \MM^{n-1}_{{\bf r}- {\bf e}_T}$.
Thus, $L \subseteq V_i$ for any $i \in T$, and therefore $L \subseteq V_1 + \cdots + V_m$.
Hence,
\[
    n-1 = \dim(U_1 + \cdots + U_m) = \dim(\pi_L(V_1 + \cdots + V_m)) = \dim(V_1 + \cdots + V_m) - 1,
\]
from which we conclude that $\dim(V_1 + \cdots + V_m) = n$ and so $(V_1,\ldots,V_m) \in \MM_{\bf r}^n$.
This proves the second claim.

Finally, we note that a lift $(V_1,\ldots,V_m)$ of $(U_1,\ldots,U_m) \in \MM^{n-1}_{{\bf r}}$ is not in $\MM_{\bf r}^n$ if and only if there exists a hyperplane $H \subseteq \mathbb{F}_q^n$ that contains each of the $V_i$.
Suppose this is the case and let $H$ be this hyperplane.
Observe that the hyperplane $H$ cannot possibly contain $L$;
otherwise, all the $U_i$ would be contained in $\pi_L(H)$, a proper subspace of $\mathbb{F}_q^n/L$.
Moreover, since $\pi_L$ induces an isomorphism $\pi_L\big|_H : H \to \mathbb{F}_q^n/L$,
the hyperplane $H$ uniquely determines the lift: $V_i = \big(\pi_L\big|_H\big)^{-1}(U_i)$ for each $i$.
There are exactly
\[ \qbinom{n}{n-1}{q} - \qbinom{n-1}{n-2}{q} = q^{n-1} \] 
hyperplanes $H \subset \mathbb{F}^n_q$ that do not contain $L$.
This establishes the third claim, and thus \cref{eq:recursion-qMs}.
\end{proof}

The special cases of \cref{lem:qMs-qmultinomial,cor:qMs m=2} suggest that $\M^n_{\bf r}(q) \in \Z_{\geq 0}[q]$. However, this positivity is not implied by \cref{prop:explicit-formula-Mq} nor the recurrence in \cref{prop:recursion-qMs}. \cref{prop:positive recurrence Ms} shows that this intuition is correct; in other words, the elements $\M^n_{\bf r}(q)$ are polynomials in $q$ with non-negative integer coefficients\footnote{In January 2025, Morales presented the conjecture that $\M^n_{\bf r}(q) \in \Z_{\geq 0}[q]$ at a talk at CAAC 2025 in Toronto, ON. Hunter Spink, who was in the audience, mentioned the conjecture to his student Matthew Bolan who proved this conjecture and communicated his proof to Morales; we include his proof for completeness.}.

\begin{proposition}[Bolan] \label{prop:positive recurrence Ms}
The coefficients $\M^n_{\bf r}(q)$ satisfy the following recurrence
\begin{equation} \label{eq:Bolan q recurrence}
\M^n_{(r_1,\ldots,r_m)}(q) = \sum_{i=0}^n q^{(n-i)(n-r_m)}\qbinom{n}{i}{q} \qbinom{i}{r_m-n+i}{q} \M^{i}_{(r_1,\ldots,r_{m-1})}(q),
\end{equation}
with base case $\M^n_{(r)}(q)=\delta_{n,r}$ for $n,r \in \Z_{\geq 0}$.
Thus $\M^n_{\bf r}(q) \in \Z_{\geq 0}[q]$.
\end{proposition}

\begin{proof}
Recall that $\M_{\bf r}^n(q)$ is the cardinality of the set
\[ \MM_{\bf r}^n :=  \Big\{ (V_1,\ldots,V_m): V_i \subset \mathbb{F}_q^{n} \textrm{ with } \dim(V_i) = r_i \textrm{ and } V_1 + \cdots + V_m = \mathbb{F}_q^{n} \Big\},\]
Consider the set
\[ \MM'= \Big\{ \big(W,(V_1,\ldots,V_{m-1})\big): \dim(W) = i, \ V_1+\cdots+V_{m-1} = W, \ \textrm{ and } \dim(V_i) = r_i \ \textrm{ for } 1 \leq i \leq m-1  \Big\}.\]
Then
\[\#\MM' = \qbinom{n}{i}{q} \cdot \M^{i}_{(r_1,\ldots,r_{m-1})}(q). \]
We claim there are
\begin{equation}\label{eq:extendingA'}
q^{(n-i)(n-r_m)}\qbinom{i}{r_m-n+i}{q} \end{equation}
ways to extend an element in $\MM'$ to one in $\MM_{\bf r}^n$, from which \cref{eq:Bolan q recurrence} will follow.

Suppose $(W, (V_1,\ldots,V_{m-1})) \in \MM'$; to extend this to an element of $\MM_{\bf r}^n$, we must pick a vector space $V_m \subseteq \F_q^n$ such that
\begin{enumerate}
    \item $\dim(V_m) = r_m$, and
    \item  $W \cap V_m = r_m-i$, since $\dim(W) = i$.
\end{enumerate}
The number of ways to pick such a $V_m$ is given by \cref{eq:extendingA'}.
\end{proof}

Finally, we identify the coefficients $\M^n_{\bf r}(q)$ as a $q$-analogue of the
coefficients $\M^n_{\bf r}$ appearing in \cref{thm:HeckeJacksonMs};
that is, these coefficients \emph{$q$-count} the same objects as $\M^n_{\bf r}$ with an appropriate statistic that we define next.

\begin{definition}
\label{def:statistic-for-Mq}
Let $[n] = \{ 1, \ldots, n\}$ and fix ${\bf r}= (r_1, \ldots, r_m)$.
\begin{enumerate}
    \item For subsets $S \subseteq U \subseteq [n]$, define
\[         \inv(S;U) := \#\big\{ (i,j) \in U^2 : i < j \,,\, i \in S \,,\, j \notin S \big\}.
\]
\item Suppose there is a tuple $(S_1,\ldots,S_m)$ of
    (not necessarily disjoint) subsets of $[{n}]$
    such that
    \begin{enumerate}
        \item $S_1\cup \cdots \cup
    S_m =[{n}]$, and
    \item $\# S_i=r_i$ for $i=1,\ldots,m$.
    \end{enumerate}
Then let $T_i:=S_1\cup \cdots \cup S_i$ (with $T_0=\varnothing$), and define
    \[
        \stat(S_1,\ldots,S_m) := \sum_{i=2}^{m} \Big(
            (\#T_i-\#T_{i-1})(\#T_i-r_i) + \inv(T_{i-1};T_i)+\inv(S_{i}\cap T_{i-1};T_{i-1})
        \Big).
    \]
\end{enumerate}
\end{definition}

\begin{example} \label{ex:stat that gives Ms}
Consider the tuple
\[ (\{1, 2, 3, 5\}, \{1, 3, 5\}, \{2, 4, 6\}) \]
counted by $\M^6_{4,3,3}$. For this tuple, we have that
\begin{align*}
    T_1 & =  \{1,2,3,5\} \\
    T_2 & = T_1 =\{1,2,3,5\}\\
    T_3 & = \{1,2,3,4,5,6\}.
\end{align*}
Hence
\begin{align*}
\stat(\{1, 2, 3, 5\}, &\{1, 3, 5\}, \{2, 4, 6\})    \\
&= 0(4-3) + \inv(T_1;T_2)+\inv(S_2\cap T_1; T_1) + 2(6-3) + \inv(T_2;T_1)+\inv(S_3\cap T_2; T_2) \\
&= 1+6+7+2= 16.
\end{align*}

\end{example}

\begin{proposition} \label{prop:M-via-q-statistic}
For a fixed ${\bf r}= (r_1,\ldots,r_m)$ and $n$, we have
    \[
    \M^n_{\bf r}(q)  = \sum_{\substack{(S_1,\ldots,S_m) \subseteq [n]^m \\ S_1 \cup \cdots \cup S_m = [n]} } q^{\stat(S_1,\ldots,S_m)}.
    \]
\end{proposition}

\begin{proof}
Recall (e.g. \cite[Prop. 1.7.1]{EC1}), that $q$-binomial coefficient can be written as a sum over $k$-subsets of $[n]$ with the statistic $\inv(S;[n])$:
\begin{equation}\label{eq:q-binomial identity}
\qbinom{n}{k}{q} = \sum_{S \in \binom{[n]}{k}} q^{\inv(S;[n])},
\end{equation}
where $\binom{[n]}{k}$ denotes the set of $k$-subsets of $[n]$. We apply \cref{eq:q-binomial identity} recursively to the right-hand-side of \cref{eq:Bolan q recurrence} to obtain
\[
\M^n_{(r_1,\ldots,r_m)}(q) = \sum_{T_{m-1}, S_m} q^{(n-\#T_{m-1})(n-r_m) + \inv(T_{m-1};T_m) + \inv(T_{m-1} \cap S_m; T_{m-1})} \M^{\#T_{m-1}}_{(r_1,\ldots,r_{m-1})}(q),
\]
where the sum is over subsets $T_{m-1}, S_m \subset [n]$ with $\#S_m = r_m$ and $T_{m-1} \cup S_m = [n]$. Iterating this $m-1$ more times gives the desired statistic.
\end{proof}

\begin{example}
The following table shows the values of $\stat$  in \cref{def:statistic-for-Mq} for the tuples counted by $\M^3_{(2,2)}$.
\begin{center}
\begin{tabular}{c|c} \hline
$(S_1,S_2)$ & $\stat$ \\
\hline
$(\{1, 2\}, \{1, 3\})$ & $4$ \\
$(\{1, 2\}, \{2, 3\})$ & $3$ \\
$(\{1, 3\}, \{1, 2\})$ & $3$ \\
$(\{1, 3\}, \{2, 3\})$ & $2$ \\
$(\{2, 3\}, \{1, 2\})$ & $2$ \\
$(\{2, 3\}, \{1, 3\})$ & $1$ \\ \hline
\end{tabular}
\end{center}
Thus, $\M^3_{(2,2)}(q) = q^4 + 2q^3 + 2q^2 + q$.
\end{example}

Calculations for $n \leq 12$ and $m\leq 15$, and the special cases of \cref{lem:qMs-qmultinomial} and $m=2$ (\cref{cor:qMs m=2}) suggest that the polynomials $\M^n_{(r_1,\ldots,r_m)}(q)$ are unimodal. 

\begin{conjecture}
\label{rem:unimodality of qMs}
For $n \in \Z_{\geq 0}$ and integers $0 \leq r_1,\ldots,r_m \leq n$, we have that $\M^n_{\bf r}(q)$ is unimodal.
\end{conjecture}

However, as in the case of the $q$-binomial coefficients, the polynomials $\M_{{\bf r}}(q)$ are not log-concave. For example $\M^4_{(2,2)}(q) = q^4 + q^5 + 2q^6 + q^7 + q^8$. Also, the polynomials $\M_{{\bf r}}(q)$ are not necessarily palindromic. For example, $\M^3_{(2,2,1)}(q)=q^6+3q^5+6q^4+6q^3+4q^2+q$.

\begin{remark}
For the $q=1$ case, the coefficients $\M_{\bf r}^n$ have a generating polynomial given by the closed formula in \cref{eq:gf for Ms}, which follows from the recurrence \eqref{eq:key recurrence M}. Since the $q$-analogue of this recurrence in \cref{prop:recursion-qMs} has an extra term, it is unclear whether there is a closed formula for the generating polynomial of  $\M_{\bf r}^n(q)$.     
\end{remark}

\section{Final remarks}
\label{sec:final-remarks}

In this section, we discuss various questions and connections that remain open, as well as directions of future work.

\subsection{Combinatorial proofs}
Our main theorems are proved algebraically. It would be very interesting to give combinatorial proofs of any of these results.

As a comparison, many of the first results on factorizations (e.g. \cite{Sta81,Jackson0,Jackson1}) were first proved algebraically using characters of the symmetric group. However, many of the leading term product formulas have elegant combinatorial proofs. For example, D\'enes \cite{Denes} gave a double counting proof of \cref{eq: tree case}, followed by bijective proofs by Moszkowski \cite{Mozkowski}. Later, Goulden and Jackson in \cite{GouldenJackson} gave a bijection from factorizations to certain {\em planar cacti} counted by the right-hand-side of \cref{eq:GJcacti}. Moreover, the identity \cref{eq:JacksonSn} for the generating polynomial of the long cycle into two factors has different bijective proofs by Schaeffer and Vassilieva \cite{SchaefferVassilieva}, Bernardi \cite{Bernardi} and Chapuy, F\'eray, and Fusy \cite{CFF}, and for $m$ factors by Bernardi and Morales \cite{BM1,BM2}. These bijective proofs are given in terms of {\em unicellular maps}, certain one face graphs embedded on orientable surfaces in correspondence with long cycle factorizations that have remarkable combinatorial bijections and decompositions (e.g. see \cite{Schaeffersurvey}). 

\subsection{Exponential generating functions for $e_{1^j}(\Xi_n(q))$ and $h_j(\Xi_n(q))$}

In \cref{thm:tree conj} and \cref{thm:cat conj} we found Hecke algebra analogues of the closed forms in \cref{eq:Jackson2} and \cref{eq:Masumoto-Novak} giving the ordinary generating functions for the numbers $a(n;j)$ and $b(n;j)$. There are also elegant closed formulas for the exponential generating functions of these numbers in \cite{Jackson1},\cite[Eq. (47),(48)]{MN}:
\begin{align}
\sum_{j\geq n-1} a(n;n-1+2g) \frac{t^{2g}}{(2g)!}  &= n^{n-2} n^{2g} \binom{n-1+2g}{n-1} \left(\frac{ \sinh(t/2)}{t/2}\right)^{n-1}, \label{eq:exp gf Jackson}\\
\sum_{j\geq n-1} b(n;n-1+2g) \frac{t^{2g}}{(2g)!}  &= C_{n-1} \binom{2n-2+2g}{2n-2}  \left(\frac{ \sinh(t/2)}{t/2}\right)^{2n-2}.
\end{align}
It would be interesting to find Hecke algebra analogues of these closed forms.

\subsection{Hecke analogues of factorizations for other permutations}

The main results of this paper concern generalizations of factorization results of the long cycle $\sfc$. However, there are other important factorization formulas for permutations other than $\sfc$, such as {\em Hurwitz numbers} \cite{Hurwitz} (see also \cite{CavalieriMiles,ALCO_GJSurvey}),  {\em monotone Hurwitz numbers} \cite{goulden2013monotone},  {\em star factorizations} \cite{IR} (see also \cite{Tenner1,Tenner2,GJTransitivePowers,LothRattan,FERAY}), and {\em weighted Hurwitz numbers} (see \cite{BenDali-thesis,ChapuyDolega,HarnadMac}). We intend to study Hecke analogues for some of these questions in future work \cite{HZHecke}.

\subsection{Connections to diagonal harmonics}
Our results for $a_q(n,n-1)$ and $b_q(n,n-1)$ give $q$ analogues of the number of trees with $n$ vertices and of the $n$th Catalan number, respectively. The $q$-analogue of the former $[n]_q^{n-2}$ has several combinatorial interpretations, for example via the multivariate Cayley theorem (see \cite[Thm.~5.13]{Bona}), one has that
\[
[n]_q^{n-2} = \sum_T  q^{\sum_{i=1}^n (i-1)(\deg_T(i)-1)},
\]
where the sum is over trees $T$ with vertices $\{1,\ldots,n\}$ and $\deg_T(i)$ is the degree of vertex $i$ in $T$.

Similarly, there are various $q$-analogues of the Catalan numbers; the $C_n(q)$ in \cref{thm:intro_q_summary_thm}(2) is due to MacMahon (see \cite{MacMahon}). The polynomial $C_n(q)$ has several interesting combinatorial interpretations, for example
\[  
C_n(q) = \sum_{w\in BW(n)}q^{\mathrm{maj}(w)} \] where $BW(n)$ is a set of binary Dyck words of length $n$ and $\maj(w)$ is the {\em major index}.
Moreover, these two $q$-analogues can be obtained from a bigraded Hilbert series related to the space of {\em diagonal harmonics} (see \cite{HaglundBook}). Indeed, if $DH_n$ is the {\em space of diagonal harmonics} of dimension $(n+1)^{n-1}$ and $DH_n^{\epsilon}$ is its {\em alternant}, a certain subspace of $DH_n$ of dimension $C_n$, then by results of Haiman \cite{HaimanDHn} and Garsia--Haiman \cite{GHaiman}, we have that 
\begin{align*}
a_q(n; n-1) &= q^{1-n}\Hilb(DH_{n-1};q,1/q),\\
b_q(n;n-1) &= q^{1-n}\Hilb(DH^{\epsilon}_{n-1};q,1/q).
\end{align*}
The latter Hilbert series is also called the {\em $(q,t)$-Catalan polynomial} (see \cite[Ch. 3]{HaglundBook}). Our $q$-analogue of the Narayana numbers in the formula $c_q(n;k,n-1-k)$ also appears in the context of diagonal harmonics \cite{qtNarayana}.
It is natural to wonder if there are connections between these bigraded Hilbert series and our work.

\subsection{Connections to the general linear group}
As mentioned in the introduction, another natural source of $q$-analogues comes from the general linear group $\GL_n(\F_q)$, where $\F_q$ is a finite field of characteristic $q$. We discuss in \cref{thm:GLn factorization} the factorization results proved for $\GL_n(\F_q)$. The relationship between $\HH_n(q)$ and $\GL_n(\F_q)$, makes it quite tempting to relate our work to the results in \cref{thm:GLn factorization}.

We can translate the factorization question in the symmetric group into the realm of $\GL_n(\F_q)$ by replacing a long cycle with a \emph{regular elliptic element} and transpositions with \emph{reflection elements}. These notions are defined as follows.
Multiplication by an element of the field $\F_q$ in $\F_{q^n}$, an $n$-dimensional vector space over $\F_q$, is  a linear transformation. Thus, there is a natural inclusion
\[ \iota: \F_{q^n}^{\times} \hookrightarrow \GL_n(\F_q). \]
An element $c \in \GL_n(\F_q)$ is said to be \emph{regular elliptic} if $c = \iota(\omega)$ for an element $\omega \in \F_{q^n}^\times$ with the property that $\{1, \omega, \omega^2, \ldots, \omega^{n-1}\}$ is a linear basis of $\F_{q^n}$ as an $\F_q$ vector space.

When $\omega$ satisfies the additional condition that it is a generator of the cyclic group $\F_{q^n}^{\times}$, then $c$ is called a {\em Singer cycle}. There is a well established analogy between long cycles in $\mathfrak{S}_n$ and Singer cycles in $\GL_n(\mathbb{F}_q)$ (see \cite[Section 9]{RSWcyclic,lewis2024Singer}). For instance,  just as a long cycle acts transitively on $\{1,\ldots,n\}$, a Singer cycle acts transitively on lines in $\mathbb{F}_q^n$. An element $t_i \in \GL_n(\mathbb{F}_q)$ is a \emph{reflection} if it fixes a hyperplane in $\F_q^n$.

Here are some results of factorizations in $\GL_n(\F_q)$ proved using characters.

\begin{theorem}\label{thm:GLn factorization}
    For $n \geq 1$ and $q = p^m$ for some prime $p$, the following holds.
    \begin{enumerate}
        \item {\rm (Lewis--Reiner--Stanton \cite{LRS}):}
Let $c$ be a regular elliptic element in $\GL_n(\mathbb{F}_q)$, then there are
\[ (q^n-1)^{n-1} \]  factorizations of $c$ into $n$ reflections. \medskip
\item {\rm (Huang--Lewis--Reiner \cite{HLR}): }  Let $c$ be a regular elliptic element in  $\GL_n(\mathbb{F}_q)$ and $\alpha=(\alpha_1,\ldots,\alpha_m)$ be a composition of $n$, then the number of factorizations of $c$ as $g_1\cdots g_m$ where $g_i \in \GL_n(\mathbb{F}_q)$ has fixed space dimension $\alpha_i$ is $q^{\epsilon(\alpha)}(q^n-1)^m$, where
\[
\epsilon(\alpha) = \sum_{i=1}^m (\alpha_i-1)(n-\alpha_i).
\]
\medskip
\item {\rm (Lewis--Morales  \cite{LMq}):}
Let $c$ be a regular elliptic element in $\GL_n(\mathbb{F}_q)$, and let $a_{r_1,\ldots,r_k}$ be the number of factorizations of $c$ as $g_1\cdots g_k$ where $g_i \in \GL_n(\mathbb{F}_q)$ has fixed space dimension $r_i$, then
\begin{equation} \label{eq:genseries-manyfactors}
     \begin{multlined}
        \frac{1}{|\GL_n(\mathbb{F}_q)|^{k - 1}}
        \sum_{r_1, \ldots, r_k} a_{r_1, \ldots, r_k}(q) \cdot x_1^{r_1} \cdots x_k^{r_k}
        \\
     \qquad   =\sum_{ \substack{{\bf p} = (p_1, \ldots, p_m) \colon \\ 0 \leq p_i \leq n}  }
        \frac{\widehat{\M}^{n - 1}_{\widehat{\bf p}}(q)}{\prod_{p \in \widehat{\bf p}} \qbinom{n-1}{p}{q}} \cdot
        \frac{(x_{1};q^{-1})_{p_1}}{(q;q)_{p_1}} \cdots \frac{(x_{k};q^{-1})_{p_k}}{(q;q)_{p_k}}.
    \end{multlined}
\end{equation}
Here $\widehat{\bf p}$ is the result of deleting all copies of $n$ from $\bf p$,
\[ \widehat{\M}^{m}_{\varnothing}(q) := 0,  \qquad \qquad (a,q)_k =(1-a)(1-aq)\cdots (1-aq^{k-1}), \]
and for $k > 0$
\begin{equation} \label{eq:def-qM}
\widehat{\M}^{m}_{r_1,\ldots,r_k}(q) := \sum_{d=0}^{\min(r_i)} (-1)^d
q^{\binom{d+1}{2}-kd} \qbinom{m}{d}{q} \prod_{i=1}^k \qbinom{m-d}{r_i-d}{q}.
\end{equation}
    \end{enumerate}
\end{theorem}
Note that the $\widehat{\M}^{m}_{r_1,\ldots,r_k}(q)$ are subtly different from our $\M^{m}_{r_1,\ldots,r_k}(q)$ (see \cref{eq:formula-Mq}). In particular, while our $\M^{m}_{r_1,\ldots,r_k}(q)$ have a combinatorial interpretation (proved in \cref{prop:M-via-q-statistic} above), there is no such interpretation for $\widehat{\M}^{m}_{r_1,\ldots,r_k}(q)$. It would be interesting to relate the two quantities.

\subsection{Hecke analogue of the Harer--Zagier formula}

There is another celebrated factorization result of the long cycle in the symmetric group called the {\em Harer--Zagier formula} \cite{HarerZagier}. The Harer--Zagier formula was used to compute a certain Euler characteristic of the moduli space of smooth curves; see \cite[Ch. 3, 4]{LandoZvonkin}.

Let $\varepsilon_g(n)$ be the number of factorizations of any long cycle $\sfc[2n] \in \symm_{2n}$ into the product \( \sfc[2n] = \sigma \pi,\)
    where $\sigma$ has cycle type $(2^n)$ (i.e. it is a fixed-point-free involution) and $\pi$ has $n+1-2g$ cycles. Harer and Zagier found the following remarkable generating polynomial and recurrence for the numbers $\epsilon_g(n)$.

\begin{theorem}[Harer--Zagier]\label{thm:HZformula}
    For a fixed nonnegative integer $n$, the numbers  $\varepsilon_g(n)$ have the following generating polynomial.
    \begin{equation} \label{eq:HZ1}
\sum_{g\geq 0} \epsilon_g(n) x^{n+1-2g} = (2n-1)!! \sum_{k=1}^{n+1} 2^{k-1} \binom{n}{k-1} \binom{x}{k}.
\end{equation}
Moreover, the numbers $\epsilon_g(n)$ satisfy the recurrence
\begin{equation} \label{eq:HZ recurrence}
(n+1)\epsilon_g(n) = 2(2n-1)\epsilon_g(n-1) +  (n-1)(2n-1)(2n-3)\epsilon_{g-1}(n-2),
\end{equation}
with initial condition $\epsilon_0(0)=1$. In particular, $\epsilon_0(n) = C_n$.
\end{theorem}

In a follow-up paper \cite{HZHecke} we find a Hecke algebra analogue of the above theorem by building on and extending the techniques of this work.

\subsection{Other types}

The formula $a(n,n-1)=n^{n-2}$ for the number of minimal factorizations of the long cycle $\ncycle[n]$ into transpositions has the following uniform formula for other reflection groups. Let $W$ be a well generated {\em complex reflection group} of rank $n$. Such groups were classified by Shephard and Todd \cite{ShephardTodd}.  Analogues of the long cycle $\ncycle$ and transpositions in $W$ are a {\em Coxeter element} and {\em reflections}, respectively. The order of a Coxeter element is known as the {\em Coxeter number} and is denoted by $h$. See for instance \cite[Section 2]{chapuy_Stump} for definitions.

Given a complex reflection group $W$ of rank $n-1$, if $a_W(k)$ is the number of factorizations of a Coxeter element in $W$ as a product of $k$ reflections, then Bessis showed in \cite[Prop. 7.6]{Bessis_Annals} that
\[ a_W(n-1)=\frac{(n-1)!h^{n-1}}{|W|}. \]
For the symmetric group $W=\mathfrak{S}_n$, which has rank $n-1$, this formula reduces to
\[ a_{\mathfrak{S}_n}(n-1)=a(n,n-1)=n^{n-2}.\] Chapuy and Stump gave in \cite{chapuy_Stump} a uniform formula for the exponential generating function for the sequence $a_W(k)$ for $k\geq n$, which specializes to \cref{eq:exp gf Jackson} in the symmetric group. The original algebraic proof was  case by case and a uniform proof was given by Douvropoulos \cite{DouvropoulosUniform}. Later, Lewis--Morales \cite[Theorem 1.3]{LM} extended \cref{thm:introsymmetricsummary}(3) to some complex reflection groups.

It would be of interest to find analogues of our results for Hecke algebras of other types.

\printbibliography

\end{document}